\documentclass[a4paper,12pt]{article}%
\usepackage[T1]{fontenc}
\usepackage[english]{babel}
\usepackage[ansinew]{inputenc}
\usepackage{lmodern}
\usepackage{amsmath,amssymb,amsfonts,amstext,amsthm}
\usepackage{graphicx}
\usepackage{xcolor,colortbl}
\usepackage{float}
\theoremstyle{definition}
\newtheorem{theorem}{Theorem}[section]
\newtheorem{lemma}{Lemma}[section]
\newtheorem{remark}{Remark}[section]

\numberwithin{equation}{section}
\begin{document}

\title{New bounds in R.S. Lehman's estimates \\for the difference $\pi\left(  x\right)  -li\left(  x\right)  $}
\author{Michael Revers}
\date{}
\maketitle

\begin{abstract}
We denote by $\pi\left(  x\right)  $ the usual prime counting function and let
$li\left(  x\right)  $ the logarithmic integral of $x$. In 1966, R.S. Lehman
came up with a new approach and an effective method for finding an upper bound
where it is assured that a sign change occurs for $\pi\left(  x\right)
-li\left(  x\right)  $ for some value $x$ not higher than this given bound. In
this paper we provide further improvements on the error terms including an
improvement upon Lehman's famous error term $S_{3}$ in his original paper. We
are now able to eliminate the lower condition for the size-length $\eta$
completely. For further numerical computations this enables us to establish
sharper results on the positions for the sign changes. We illustrate with some
numerical computations on the lowest known crossover regions near $10^{316}$
and we discuss numerically on potential crossover regions below this value.



\end{abstract}

\section{Introduction}

\noindent Let $\pi\left(  x\right)  $ denote the number of primes less than or
equal to $x$ and denote the logarithmic integral by%
\[
li\left(  x\right)  =\int_{0}^{x}\frac{1}{\log t}dt.
\]
In 1791, Gauss conjectured that $\pi\left(  x\right)  $ behaves asymptotically
like $x/\log\left(  x\right)  $ as $x$ tends to infinity. Later in 1849, Gauss
suggested that the logarithmic integral $li\left(  x\right)  $ gives a much
better approximation for $\pi\left(  x\right)  .$ Gauss also noted that the
logarithmic integral constantly overestimates the prime counting function,
that is $\pi\left(  x\right)  < li\left(  x\right)  $. Gauss checked this for
any $x$ in the interval $\left[  2, 3000000 \right]  $. Large computations
have now shown that $\pi\left(  x\right)  < li\left(  x\right)  $ holds for
all $x \leq10^{19}$, see (\cite{B\"uthe}, Theorem 2). For a long period it was
conjectured that $\pi\left(  x\right)  < li\left(  x\right)  $ holds for all
$x\geq2.$ However, in 1914, Littlewood \cite{Littlewood} managed to prove that
$\pi\left(  x\right)  -li\left(  x\right)  $ changes sign infinitely often. In
fact he proved that for some positive constant $c$ there are arbitrarily large
values of $x$ such that
\[
\pi\left(  x\right)  -li\left(  x\right)  >c\frac{\sqrt{x}\log\log\log x}{\log
x}\text{ and }\pi\left(  x\right)  -li\left(  x\right)  <-c\frac{\sqrt{x}%
\log\log\log x}{\log x}.
\]
Unfortunately, Littlewood's proof did not provide an effective method for
determining or estimating an $x$ where we may expect a sign change or in other
words a crossover for the difference $\pi\left(  x\right)  -li\left(
x\right)  .$ In 1933, Skewes \cite{Skewes1}, depending on the validity of the
Riemann hypothesis, has obtained a first effective upper bound for a crossover
at a value $10^{10^{10^{34}}}.$ Later \cite{Skewes2} in 1955 he obtained an
unconditional bound at a value $10^{10^{10^{964}}}$. In 1966, R.S. Lehman
\cite{Lehman} came up with a completely new approach, a so called integrated
version of Riemann's explicit formula. Roughly speaking, his idea was to
integrate the function $u\rightarrow\pi\left(  e^{u}\right)  -li\left(
e^{u}\right)  $ against a Gaussian kernel over a carefully chosen interval
$\left[  \omega-\eta, \omega+\eta\right]  $. The resulting definite integral
is then denoted by $I\left(  \omega,\eta\right)  $. Lehman's method amounts
then to finding suitable values $\omega$ and $\eta$ such that this integral
becomes a positive value, that is
\[
I\left(  \omega,\eta\right)  >0.
\]
Now, by virtue of the positivity of the Gaussian kernel, the term $\pi\left(
e^{u}\right)  -li\left(  e^{u}\right)  $ must admit some positive values for
$u$ in the above mentioned search interval. In \cite{Lehman}, Lehman succeeded
to prove that there must be a crossover near $10^{1165}$. Thus, he brought
down considerably the formerly obtained bounds from Skewes. Since that time
improvements and refinements regarding error estimates and search regions have
been obtained, beginning by te Riele \cite{teRiele} in 1987, Bays and Hudson
\cite{BaysandHudson} in 2000, Plymen and Chao \cite{ChaoandPlymen} in 2010,
Saouter and Demichel \cite{SaouterandDemichel} in 2010, and the latest
available result by Saouter, Demichel and Trudgian
\cite{SaouterandDemichelandTrudgian} in 2015. See some details on the history
and the locations of the crossover regions in Table \ref{historytable}. It is
not known whether we have still touched the first sign change at
$x\simeq1.39\times10^{316}$. The only paper which goes into some details of
numerical investigations for earlier candidates is that by Bays and Hudson,
were they present the first available computer plots of $\pi\left(  x\right)
-li\left(  x\right)  $ in logarithmic scale for values of $x$ from $10^{6}$ up
to $10^{400}$ by plotting over $10000$ values. See (\cite{BaysandHudson},
Figure 1a, 1b). Bays and Hudson obtained possible earlier candidates for
crossovers near $10^{176},10^{179},10^{190},10^{260},10^{298}$ together with
some other possible but minor attractive candidates. For a short discussion on
this topic we refer to section \ref{Near crossovers}.

\section{R.S. Lehman's result and further improvements}

\subsection{The result of R.S. Lehman}

\noindent We begin by presenting the following theorem of Lehman
\cite{Lehman}. As stated before, to consider an integrated version for
$\pi\left(  x\right)  -li\left(  x\right)  $ was an innovative idea. Moreover
Lehman's method did not assume the Riemann hypothesis in its full strength but
only up to a certain height $A$. Therefore his theorem contains an additional
extra error term, the $S_{6}$ term, in case the Riemann hypothesis is not true
in general. As a final remark, Lehman's method takes use of the computation of
the zeta zeros up to a height $T\leq A$ inside the critical strip. We have

\begin{theorem}
\label{theorem21} Let $A$ be a positive number such that $\beta=\frac{1}{2}$
for all zeros $\rho=\beta+i\gamma$ of $\zeta\left(  s\right)  $ for which
$0<\gamma\leq A$. Let $\alpha,\eta$ and $\omega$ be positive numbers such that
$\omega-\eta>1$ and the conditions%
\begin{align}
\frac{4A}{\omega}  &  \leq\alpha\leq A^{2}\text{,}\label{LehmanCondition1}\\
\frac{2A}{\alpha}  &  \leq\eta<\frac{\omega}{2} \label{LehmanCondition2}%
\end{align}
hold. Let%
\begin{align*}
K\left(  x\right)   &  =\sqrt{\frac{\alpha}{2\pi}}e^{-\frac{\alpha}{2}x^{2}%
},\\
I\left(  \omega,\eta\right)   &  =\int_{\omega-\eta}^{\omega+\eta}K\left(
u-\omega\right)  ue^{-u/2}\left(  \pi\left(  e^{u}\right)  -li\left(
e^{u}\right)  \right)  du.
\end{align*}
Then, for $2\pi e<T\leq A,$%
\[
I\left(  \omega,\eta\right)  =-1-\sum_{\substack{\rho\\0<\left\vert
\gamma\right\vert \leq T}}\frac{e^{i\gamma\omega}}{\rho}e^{-\gamma^{2}%
/2\alpha}+R
\]
where $\left\vert R\right\vert \leq S_{1}+S_{2}+S_{3}+S_{4}+S_{5}+S_{6}$ with%
\begin{align*}
S_{1}  &  =\frac{3}{\omega-\eta}+4\left(  \omega+\eta\right)  e^{-\frac
{\omega-\eta}{6}},\\
S_{2}  &  =\frac{2e^{-\frac{\alpha}{2}\eta^{2}}}{\eta\sqrt{2\pi\alpha}},\\
S_{3}  &  =0.08\sqrt{\alpha}e^{-\frac{\alpha}{2}\eta^{2}},\\
S_{4}  &  =e^{-\frac{T^{2}}{2\alpha}}\left(  \frac{\alpha}{\pi T^{2}}\log
\frac{T}{2\pi}+\frac{8\log T}{T}+\frac{4\alpha}{T^{3}}\right)  ,\\
S_{5}  &  =\frac{0.05}{\omega-\eta},\\
S_{6}  &  =A\log Ae^{-\frac{A^{2}}{2\alpha}+\frac{\omega+\eta}{2}}\left(
\frac{4}{\sqrt{\alpha}}+15\eta\right)  .
\end{align*}
If the Riemann hypothesis holds, then conditions (\ref{LehmanCondition1}) and
(\ref{LehmanCondition2}) and the term $S_{6}$ (this is Lehman's original
$S_{3}$ term) in the estimate for $R$ may be omitted.
\end{theorem}

\noindent As one easily sees, there are a lot of parameters involved. We
explain them briefly.

\begin{table}[H]
\begin{center}
{\footnotesize
\begin{tabular}
[c]{|r|l|}\hline
\cellcolor{orange}$\omega$ & Exponent of the center location of the
interval.\\
\cellcolor{orange}$\eta$ & Exponent of the radius of the interval, defining
its sharpness.\\
\cellcolor{orange}$\alpha$ & Peak behavior of the Gaussian kernel.\\
\cellcolor{orange}$A$ & Height on the imaginary axis for which the Riemann
hypothesis holds true.\\
\cellcolor{orange}$T$ & Height on the imaginary axis for which the zeta zeros
are computed.\\\hline
\end{tabular}
}
\end{center}
\caption{Parameters in Lehman's theorem.}%
\end{table}

\subsection{Improvements}

It was shown by te Riele \cite{teRiele} that there is a crossover near $6.663
\times10^{370}$. He used improved parameters and 50000 zeta zeros. Later Bays
and Hudson \cite{BaysandHudson} did similar with $10^{6}$ zeros, 20 times that
of te Riele. They succeeded in showing a crossover near $1.398 \times10^{316}%
$. Chao and Plymen \cite{ChaoandPlymen} were the first to make modifications
to Lehman's theorem in reducing the constant in the critical leading $S_{1}$
and $S_{5}$ term from 3.05 to 2.1111. Later improvements were given by Saouter
and Demichel \cite{SaouterandDemichel} by using an improved estimate for the
prime counting function established by Dusart, see
(\cite{SaouterandDemichelandTrudgian}, Theorem 2.4). Saouter, Demichel and
Trudgian \cite{SaouterandDemichelandTrudgian} improved further by modifying
the weight function in the integral against the Gaussian kernel. In Figure
\ref{historyfigure} we present a visualization on the improved crossover
regions from the data in Table \ref{historytable}. The horizontal lines
represent the intervals $\left[  \omega-\eta, \omega+\eta\right]  $.

\begin{table}[H]
\begin{center}
{\scriptsize
\begin{tabular}
[c]{|r|r|r|r|r|r|}\hline
\rowcolor{orange} Name & $\omega$ & $\eta$ & $\alpha$ & $A$ & Zeta
Zeros\\\hline
R.S. Lehman, 1966 & 2682.976800000 & 0.0340000000 & $10^{7}$ & 170000 &
12500\\
H. te Riele, 1987 & 853.852286000 & 0.0045000000 & $2 \times10^{8}$ & 450000 &
50000\\
Bays, Hudson, 1999 & 727.952088813 & 0.0020000000 & $10^{10}$ & $10^{7}$ &
$10^{6}$\\
Plymen, Chao, 2010 & 727.952018000 & 0.0001600000 & $1.3 \times10^{11}$ &
$1.02 \times10^{7}$ & $2 \times10^{6}$\\
Saouter, Demichel, 2010 & 727.951335792 & 0.0000228333 & $6 \times10^{12}$ &
$6.85 \times10^{7}$ & $2.2 \times10^{7}$\\
Trudgian et al, 2015 & 727.951335426 & 0.0000014100 & $1.6 \times10^{15}$ &
$1.13 \times10^{9}$ & $5.25 \times10^{8}$\\\hline
\end{tabular}
}
\end{center}
\caption{Improvements on the crossover regions.}%
\label{historytable}%
\end{table}

\begin{figure}[H]
\begin{center}
\includegraphics[width=\textwidth]{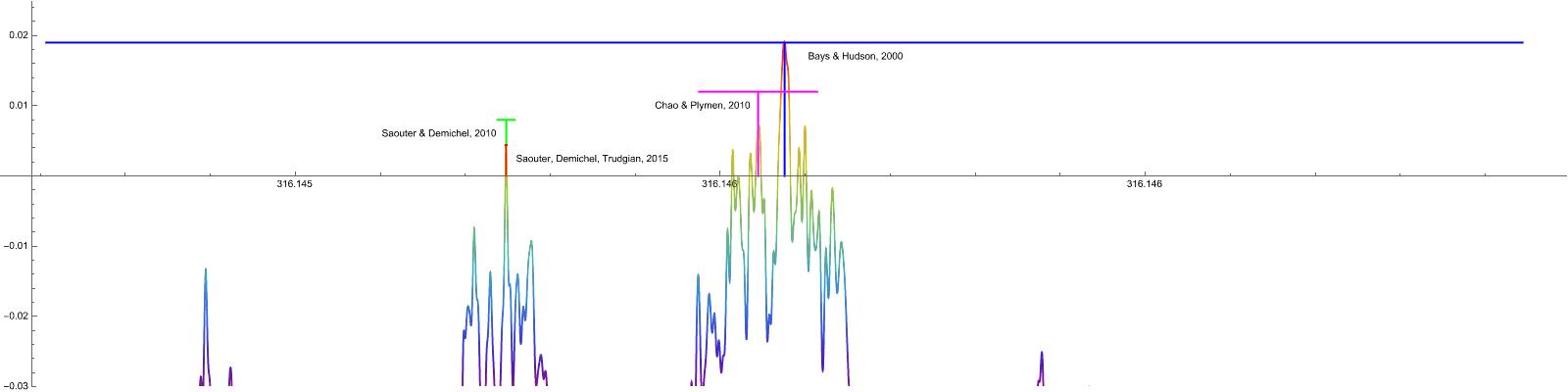}
\end{center}
\caption{Improvements in the crossover regions near $\omega= 727.95$.}
\label{historyfigure}
\end{figure}

\medskip\noindent In order to keep on improvements which are directly
comparable to Lehman's result, we present the results of Saouter and Demichel
(\cite{SaouterandDemichel}, Theorem 3.2) from 2010.

\begin{theorem}
\label{theorem22} Let $A$ be a positive number such that $\beta=\frac{1}{2}$
for all zeros $\rho=\beta+i\gamma$ of $\zeta\left(  s\right)  $ for which
$0<\gamma\leq A$. Let $\alpha,\eta$ and $\omega$ be positive numbers such that
$\omega-\eta>25.57$ and the conditions%
\begin{align}
\frac{4A}{\omega}  &  \leq\alpha\leq A^{2}\text{,}\label{Condition1}\\
\frac{2A}{\alpha}  &  \leq\eta<\frac{\omega}{2} \label{Condition2}%
\end{align}
hold. Let%
\begin{align*}
K\left(  x\right)   &  =\sqrt{\frac{\alpha}{2\pi}}e^{-\frac{\alpha}{2}x^{2}%
},\\
I\left(  \omega,\eta\right)   &  =\int_{\omega-\eta}^{\omega+\eta}K\left(
u-\omega\right)  ue^{-u/2}\left(  \pi\left(  e^{u}\right)  -li\left(
e^{u}\right)  \right)  du.
\end{align*}
Then, for $2\pi e<T\leq A,$%
\[
I\left(  \omega,\eta\right)  =-1-\sum_{\substack{\rho\\0<\left\vert
\gamma\right\vert \leq T}}\frac{e^{i\gamma\omega}}{\rho}e^{-\gamma^{2}%
/2\alpha}+R
\]
where $\left\vert R\right\vert \leq S_{1}^{\prime}+S_{2}+S_{3}+S_{4}%
+S_{5}+S_{6}$ with%
\[
S_{1}^{\prime}=\frac{2}{\omega-\eta}+\frac{10.04}{\left(  \omega-\eta\right)
^{2}}+\log2\left(  \omega+\eta\right)  e^{-\frac{\omega-\eta}{2}}+\frac
{2}{\log2}\left(  \omega+\eta\right)  e^{-\frac{\omega-\eta}{6}}%
\]
and the other terms remain unchanged. If the Riemann hypothesis holds, then
conditions (\ref{Condition1}) and (\ref{Condition2}) and the term $S_{6}$ in
the estimate for $R$ may be omitted.
\end{theorem}

\begin{remark}
Roger Plymen kindly informed me that by an inequality of Dusart \cite{Dusart} 
from 2018 one can improve the $S_{1}^{\prime}$ to a better form. This
refinement of Lehman's theorem can also be found in (\cite{Plymen}, Theorem
6.6). To obtain the best possible results, we incorporate this improvement in
the subsequent theorem.
\end{remark}

\section{Main theorem}

We present the following further improvements.

\begin{theorem}
\label{theorem31} Let $A$ be a positive number such that $\beta=\frac{1}{2}$
for all zeros $\rho=\beta+i\gamma$ of $\zeta\left(  s\right)  $ for which
$0<\gamma\leq A$. Let $\alpha,\eta$ and $\omega$ be positive numbers such that
$w-\eta\geq44.22$ and the conditions%
\begin{align}
\frac{5A}{4\omega}  &  \leq\alpha\leq A^{2},\label{Condition3}\\
0  &  <\eta<\frac{\omega}{100} \label{Condition4}%
\end{align}
hold. Let%
\begin{align*}
K\left(  x\right)   &  =\sqrt{\frac{\alpha}{2\pi}}e^{-\frac{\alpha}{2}x^{2}%
},\\
I\left(  \omega,\eta\right)   &  =\int_{\omega-\eta}^{\omega+\eta}K\left(
u-\omega\right)  ue^{-u/2}\left(  \pi\left(  e^{u}\right)  -li\left(
e^{u}\right)  \right)  du.
\end{align*}
Then, for $2\pi e<T\leq A,$%
\[
I\left(  \omega,\eta\right)  \geq-1-\sum_{\substack{\rho\\0<\left\vert
\gamma\right\vert \leq T}}e^{i\gamma\omega}\left(  \frac{1}{\rho}+\frac
{1}{\omega\rho^{2}}\right)  e^{-\gamma^{2}/2\alpha}+R,
\]
where $\left\vert R\right\vert \leq R_{1}+R_{2}+R_{3}+R_{4}+R_{5}+R_{6}$ with%
\begin{align*}
R_{1}  &  =\frac{2}{\omega-\eta}+\frac{8}{\left(  \omega-\eta\right)  ^{2}%
}+\frac{58.56}{\left(  \omega-\eta\right)  ^{3}}+\log2\left(  \omega
+\eta\right)  e^{-\frac{\omega-\eta}{2}}+\frac{2\left(  \omega+\eta\right)
e^{-\frac{\omega-\eta}{6}}}{\log2},\\
R_{2}  &  =\frac{0.037}{\omega\sqrt{\alpha}}\left(  \min\left\{
0.082\alpha,\frac{1}{\eta}\right\}  e^{-\frac{\alpha}{2}\eta^{2}}+\frac
{1}{\omega-\eta}\left(  1-e^{-\frac{\alpha}{2}\eta^{2}}\right)  \right)  ,\\
R_{3}  &  =0.074\sqrt{\alpha}e^{-\frac{\alpha}{2}\eta^{2}},\\
R_{4}  &  =e^{-\frac{T^{2}}{2\alpha}}\left(  \frac{\alpha}{\pi T^{2}}\log
\frac{T}{2\pi}+\frac{8\log T}{T}+\frac{4\alpha}{T^{3}}\right)  \left(
1+\frac{1}{\omega T}\right)  ,\\
R_{5}  &  =\frac{0.003}{\left(  \omega-\eta\right)  ^{2}},\\
R_{6}  &  =\left(  1+\frac{22}{Aw}\right)  A\log A\left(  \frac{8.283}%
{A}e^{-\frac{\alpha}{4}\eta^{2}+\frac{\omega+\eta}{2}}+7.152\eta
e^{-\frac{A^{2}}{2\alpha}+\frac{\omega+\eta}{2}}\right)  ,
\end{align*}
If the Riemann hypothesis holds, then conditions (\ref{Condition3}) and
(\ref{Condition4}) and the term $R_{6}$ in the estimate for $R$ may be omitted.
\end{theorem}

\noindent Before we proceed further, we will give some remarks.

\begin{remark}
The improvements in Theorem \ref{theorem31} do not come for free. In the
summation for $I\left(  \omega,\eta\right)  $ we have to do some extra work by
summing also over the $1/\rho^{2}$ terms. However, this does not significantly
affect any computer running time.
\end{remark}

\begin{remark}
The new $R_{1}$ term is derived from (\cite{Plymen}, Theorem 6.6).
\end{remark}

\begin{remark}
The new $R_{2}$ term is induced by the double summation process and represents
now the previous $S_{2}$ term. Also, the $R_{4}$ term is induced by the double
summation process.
\end{remark}

\begin{remark}
In condition (\ref{Condition3}) we decrease the lower bound for $\alpha$ by a
factor of $5/16$ compared to (\ref{Condition1}) and in condition
(\ref{Condition4}) we eliminate the full left-hand restriction on $\eta$
compared to (\ref{Condition2}). At the same time the upper restriction in
(\ref{Condition4}) appears much stronger than in (\ref{Condition2}). However,
in all practical purposes this restriction does not play a relevant
computational role, since we are always interested in smallest possible values
for $\eta$. Compare this in Table \ref{historytable} where one sees that the
ratio $\eta/\omega$ is usually much smaller than $10^{-5}$.
\end{remark}

\begin{remark}
Via the double summation process we can reduce the $S_{5}$ term in Theorem
\ref{theorem22} into $R_{5}$ even by a higher order. Note, that error values
of size $\left(  \omega-\eta\right)  ^{-1}$ are part of the most critical
values in the error terms.
\end{remark}

\begin{remark}
Note that the $R_{6}$ term is of better order in its first exponential term
when we insert the original restriction $\eta\geq2A / \alpha$. Let us take a
closer look on this, when we insert the minimal possible lower bound $\eta= 2A
/ \alpha$ inside the original $S_{6}$ term now compared to the new $R_{6}$
term. We then have
\begin{align*}
S_{6}  &  = A\log Ae^{-\frac{A^{2}}{2\alpha}+\frac{\omega+\eta}{2}}\left(
\frac{4}{\sqrt{\alpha}}+15\eta\right) \\
&  = A\log A\left(  \frac{4}{\sqrt{\alpha}}e^{-\frac{\alpha}{8}\eta^{2}%
+\frac{\omega+\eta}{2}}+15\eta e^{-\frac{A^{2}}{2\alpha}+\frac{\omega+\eta}%
{2}}\right)  .
\end{align*}
Note the value $8$ in the exponent denominator as well as the $\sqrt{\alpha}$
term in the former denominator and keep in mind that we also have to certify
$\sqrt{\alpha} \leq A$, which comes from condition (\ref{Condition3}). Now we
have
\[
R_{6} =\left(  1+\frac{22}{Aw}\right)  A\log A\left(  \frac{8.283}{A}%
e^{-\frac{\alpha}{4}\eta^{2}+\frac{\omega+\eta}{2}}+7.152\eta e^{-\frac{A^{2}%
}{2\alpha}+\frac{\omega+\eta}{2}}\right)  ,
\]
without any further restriction on $\eta$ regarding the lower bound. Also
note, that the large numerator $A$ cancels out completely in the first exp
term. Now, looking to the new exponent denominator value $4$ in the $R_{6}$
term it becomes evident that we can tighten the value $\eta$ in $R_{6}$ at
least by a factor $1/\sqrt{2}$, thus arriving on the comparable size in the
$S_{6}$ term without any relevant loss in accuracy. Finally, the leading
magnification factor $1+22/A\omega$ is so close to the value $1$ that it does
not touch any relevant machine size accuracy.
\end{remark}

\begin{remark}
Looking to the result of Saouter, Trudgian and Demichel
(\cite{SaouterandDemichelandTrudgian}, Theorem 3.1) from 2015 they also used
the same double summation formula for computing their $I\left(  \omega
,\eta\right)  $ value. However, they defined $I\left(  \omega,\eta\right)  $
with a different kernel function, that means by
\[
I\left(  \omega,\eta\right)  =\int_{\omega-\eta}^{\omega+\eta}K\left(
u-\omega\right)  \frac{ue^{-u/2}}{1+\frac{2}{u}+\frac{10.04}{u^{2}}}\left(
\pi\left(  e^{u}\right)  -li\left(  e^{u}\right)  \right)  du.
\]
A natural question is whether their estimates can also be generalized in the
same way, especially their $R_{5}$ error term which is the corresponding term
to Lehman's original $S_{3}$ term, or to the $S_{6}$ term in Theorem
\ref{theorem21}. Here we can follow more or less the same bounding strategy as
in our proof. For more details we refer to section \ref{Trudgian}.
\end{remark}

\medskip\noindent The rest of the paper is organized as follows. First, to
keep the paper self-contained as far as possible, we proceed through the
proofs for the error terms in Theorem \ref{theorem31} by presenting the
relevant details, though we note that parts of calculations are already
published in other sources. In section \ref{Preliminaries} we start with some
preliminaries. In section \ref{ErrorTerms} we present the proofs for the error
terms $R_{1}$ to $R_{6}$. In section \ref{Numerical applications} we
demonstrate some numerical applications for the new bounds. In section
\ref{Trudgian} we discuss the corresponding refinement in the
Saouter-Trudgian-Demichel result from 2015. At least, in section
\ref{Near crossovers} we give a short discussion for near crossovers earlier
than $10^{316}$ together with some computational results and figures.

\section{Preliminaries}

\label{Preliminaries}

\subsection{The Riemann-von Mangoldt formula}

Let
\begin{equation}
\Pi\left(  x\right)  =\pi\left(  x\right)  +\frac{1}{2}\pi\left(
x^{1/2}\right)  +\frac{1}{3}\pi\left(  x^{1/3}\right)  +\cdots\label{Pipizero}%
\end{equation}
and let
\begin{align}
\Pi_{0}\left(  x\right)   &  =\lim_{\varepsilon\rightarrow0}\frac{1}{2}\left(
\Pi\left(  x+\varepsilon\right)  +\Pi\left(  x-\varepsilon\right)  \right)
,\nonumber\\
\pi_{0}\left(  x\right)   &  =\lim_{\varepsilon\rightarrow0}\frac{1}{2}\left(
\pi\left(  x+\varepsilon\right)  +\pi\left(  x-\varepsilon\right)  \right)
\label{Pizero}\\
&  =\left\{
\begin{array}
[c]{ll}%
\pi\left(  x\right)  -1/2, & x\text{ is prime,}\\
\pi\left(  x\right)  , & \text{otherwise.}%
\end{array}
\right. \nonumber
\end{align}
Since the right hand side in (\ref{Pipizero}) is a finite sum consisting of
$\left[  \log x/\log2\right]  $ terms we have%
\begin{equation}
\Pi_{0}\left(  x\right)  =\pi_{0}\left(  x\right)  +\frac{1}{2}\pi_{0}\left(
x^{1/2}\right)  +\frac{1}{3}\pi_{0}\left(  x^{1/3}\right)  +\cdots
\label{defPizero}%
\end{equation}
and%
\begin{align}
\frac{1}{3}\pi_{0}\left(  x^{1/3}\right)  +\frac{1}{4}\pi_{0}\left(
x^{1/4}\right)  +\cdots &  \leq\frac{1}{3}\pi_{0}\left(  x^{1/3}\right)
\left[  \frac{\log x}{\log2}\right] \nonumber\\
&  \leq\frac{1}{3}\pi\left(  x^{1/3}\right)  \frac{\log x}{\log2}.
\label{ineq01}%
\end{align}
The Riemann-von Mangoldt formula now states that for $x>1,$%
\begin{equation}
\Pi_{0}\left(  x\right)  =li\left(  x\right)  -\sum_{\rho}li\left(  x^{\rho
}\right)  +\int_{x}^{\infty}\frac{du}{\left(  u^{2}-1\right)  u\log u}-\log2,
\label{Mangoldt}%
\end{equation}
where $\rho$ runs over the zeta zeros inside the critical strip in increasing
order of $\left\vert \operatorname{Im}\left(  \rho\right)  \right\vert $ and
$li\left(  x^{\rho}\right)  $ is to be interpreted as $li\left(  e^{\rho\log
x}\right)  $. From (\ref{Pizero}), (\ref{defPizero}), (\ref{ineq01}) and
(\ref{Mangoldt}) we get for $x>1$ the estimate%
\begin{align}
\pi\left(  x\right)  -li\left(  x\right)   &  \geq-\frac{1}{2}\pi\left(
x^{1/2}\right)  -\sum_{\rho}li\left(  x^{\rho}\right)  -\frac{1}{3}\pi\left(
x^{1/3}\right)  \frac{\log x}{\log2}\nonumber\\
&  +\int_{x}^{\infty}\frac{du}{\left(  u^{2}-1\right)  u\log u}-\log2.
\label{Ineqpili}%
\end{align}
Now we use Dusart's estimate (\cite{Dusart}, Theorem 5.1) for $\pi\left(
x\right)  $ together with the classic bound%
\begin{align*}
\pi\left(  x\right)   &  \leq\frac{x}{\log x}\left(  1+\frac{1}{\log x}%
+\frac{2}{\log x}+\frac{7.32}{\log^{3}x}\right)  ,\text{ if }x\geq4\cdot
10^{9},\\
\pi\left(  x\right)   &  \leq\frac{2x}{\log x},\text{ if }x>1,
\end{align*}
to obtain from (\ref{Ineqpili})%
\begin{align}
\pi\left(  x\right)  -li\left(  x\right)   &  \geq-\sum_{\rho}li\left(
x^{\rho}\right)  -\frac{x^{1/2}}{\log x}\left(  1+\frac{2}{\log x}+\frac
{8}{\log^{2}x}+\frac{58.56}{\log^{3}x}\right) \nonumber\\
&  -\frac{2x^{1/3}}{\log2}+\int_{x}^{\infty}\frac{du}{\left(  u^{2}-1\right)
u\log u}-\log2. \label{ineq02}%
\end{align}
Since $44.22>\log\left(  \left(  4\cdot10^{9}\right)  ^{2}\right)  $ we obtain
from (\ref{ineq02}) together with $x=e^{u}$ for $u\geq44.22$ a first formula
for our further purposes, namely%
\begin{align}
ue^{-u/2}\left(  \pi\left(  e^{u}\right)  -li\left(  e^{u}\right)  \right)
&  \geq-ue^{-u/2}\sum_{\rho}li\left(  e^{u\rho}\right)  -1-\frac{2}{u}%
-\frac{8}{u^{2}}-\frac{58.56}{u^{3}}\nonumber\\
&  -ue^{-u/2}\log2-\frac{2ue^{-u/6}}{\log2}. \label{ineq03}%
\end{align}

\subsection{The function $li\left(  x \right)  $}

For $z=x+iy,y\neq0$ the logarithmic integral is defined (\cite{Ingham}, page
82) by%
\begin{equation}
li\left(  e^{z}\right)  =\int_{-\infty+iy}^{x+iy}\frac{e^{t}}{t}dt=e^{z}%
\int_{0}^{\infty}\frac{e^{-t}}{z-t}dt. \label{Complexli}%
\end{equation}
If $li\left(  e^{z}\right)  $ is defined for real $z=x$ as $\frac{1}{2}%
\lim_{y\rightarrow0}\left(  li\left(  e^{x+iy}\right)  +li\left(
e^{x-iy}\right)  \right)  $, it is easily verified that we recover the
classical definition of $li\left(  x\right)  $, where for $x>1$ it is
considered as the integral principal value%
\[
li\left(  x\right)  =\lim_{\varepsilon\rightarrow0}\left(  \int_{0}%
^{1-\varepsilon}\frac{dt}{\log t}+\int_{1+\varepsilon}^{x}\frac{dt}{\log
t}\right)  .
\]
Now, we keep to important facts in mind.

\noindent Firstly, from (\ref{Complexli}) it follows (\cite{Lebedev}, pages
30, 38-41) that $li\left(  e^{z}\right)  $ is an analytic function in every
sector $0<\delta\leq\left\vert \arg\left(  z\right)  \right\vert \leq
\pi-\delta<\pi$ in the complex plane.

\noindent Secondly, for $n\geq1,z\in\mathbb{C}$ with $\operatorname{Im}\left(
z\right)  \neq0$ we have%
\begin{equation}
li\left(  e^{z}\right)  =e^{z}\left(  \sum_{k=1}^{n}\frac{\left(  k-1\right)
!}{z^{k}}+n!\int_{0}^{\infty}\frac{e^{-t}}{\left(  z-t\right)  ^{n+1}%
}dt\right)  . \label{LiLehman}%
\end{equation}
This was proved by Lehman (\cite{Lehman}, formula 4.5) for the case $n=1$.
Successive integration by parts establishes the result.

\subsection{Further preliminaries}

Subsequently we denote by $K\left(  x\right)  =\sqrt{\frac{\alpha}{2\pi}%
}e^{-\frac{\alpha}{2}x^{2}}$ for $\alpha> 0$. Further denote by $\rho
=\beta+i\gamma$ the zeta zeros inside the critical strip.

\begin{lemma}
\label{Kxequation1} For $c \in\mathbb{R}$ we have
\[
\int_{-\infty}^{\infty}K\left(  x\right)  e^{icx}dx=e^{-\frac{c^{2}}{2\alpha}%
}.
\]

\end{lemma}

\begin{proof}
\[
\int_{-\infty}^{\infty}K\left(  x\right)  e^{icx}dx=\sqrt{\frac{\alpha}{2\pi}%
}\int_{-\infty}^{\infty}e^{-\frac{\alpha}{2}\eta^{2}}\cos\left(  cx\right)
dx=e^{-\frac{c^{2}}{2\alpha}},
\]
where, in the last equation, we have used (\cite{Gradshteyn}, formula 3.896.2) with values $q^{2}=\alpha/2, p=c$ and $\lambda=0$.
\end{proof}

\begin{lemma}
\label{Kxequation2} For $\eta>0$ we have
\[
\int_{0}^{\eta}xK\left(  x\right)  dx=\frac{1}{\sqrt{2\pi\alpha}}\left(
1-e^{-\frac{\alpha}{2}\eta^{2}}\right)  .
\]

\end{lemma}

\begin{proof}
Using $K^{\prime}\left(  x\right)  =-\alpha xK\left(  x\right)$ the result
is easy to determine.
\end{proof}

\begin{lemma}
\label{Kxestimate2} Let $c>0$ and $h:\left[  c,\infty\right)  \rightarrow
\left(  0,\infty\right)  $ monotone decreasing. Then
\[
\int_{c}^{\infty}h\left(  x\right)  e^{-\frac{x^{2}}{2\alpha}}dx\leq
\frac{\alpha}{c}h\left(  c\right)  e^{-\frac{c^{2}}{2\alpha}}.
\]

\end{lemma}

\begin{proof}
Find a proof in (\cite{Lehman}, page 401, Lemma 4).
\end{proof}

\begin{lemma}
\label{Kxestimate3} Let $c,\eta>0$ \ Then%
\[
\left\vert \int_{\eta}^{\infty}K\left(  x\right)  e^{icx}dx\right\vert \leq
K\left(  \eta\right)  \min\left\{  \frac{2}{c},\frac{1}{\alpha\eta}\right\}
.
\]

\end{lemma}

\begin{proof}
The first inequality is proved in (\cite{Lehman}, page 403). For the second inequality we simply
obtain from the triangle inequality
\begin{align*}
\left\vert \int_{\eta}^{\infty}K\left(  x\right)  e^{icx}dx\right\vert  &
\leq\int_{\eta}^{\infty}K\left(  x\right)  dx\\
& =\frac{1}{\sqrt{2\pi\alpha}}\int_{\alpha\eta}^{\infty}e^{-\frac{x^{2}
}{2\alpha}}dx.
\end{align*}
Now, using Lemma \ref{Kxestimate2} together with $h\left( x\right) =1$, we obtain the desired inequality.
\end{proof}

\begin{lemma}
\label{Kxequation3} Let $c\in\mathbb{R}$ and $\eta>0$. Then
\[
\int_{-\eta}^{\eta}K\left(  x\right)  e^{icx}dx=e^{-\frac{c^{2}}{2\alpha}%
}-2\int_{\eta}^{\infty}K\left(  x\right)  \operatorname{Re}\left(
e^{icx}\right)  dx.
\]

\end{lemma}

\begin{proof}
By some standard arguments and using Lemma \ref{Kxequation1} the result is easy established.
\end{proof}

\begin{lemma}
\label{backlundestimate} Let $2\pi e\leq T_{1}<T_{2}$ and $f:\left[
T_{1},T_{2}\right]  \rightarrow\left(  0,\infty\right)  $ monotone decreasing.
Then for some $\left\vert \delta\right\vert \leq1$ we have
\[
\sum_{T_{1} \leq\gamma\leq T_{2}}f\left(  \gamma\right)  =\frac{1}{2\pi}%
\int_{T_{1}}^{T_{2}}f\left(  x\right)  \log\frac{x}{2\pi}dx+\delta\left(
4f\left(  T_{1}\right)  \log T_{1}+2\int_{T_{1}}^{T_{2}}\frac{f\left(
x\right)  }{x}dx\right)  .
\]

\end{lemma}

\begin{proof}
Lemma \ref{backlundestimate} is proved using the classical estimate on $N\left(  T\right)$, the
number of zeros for which $0<\gamma\leq T$, obtained by Backlund \cite{Backlund}. A full proof for Lemma \ref{backlundestimate} can be found in \cite{Lehman}.
\end{proof}

\begin{lemma}
\label{gammaestimate} We have

\noindent(a) Let $T\geq2\pi e$ and $n\geq2$. Then%
\[
\sum_{T<\gamma}\frac{1}{\gamma^{n}}<T^{1-n}\log T.
\]
(b)%
\begin{align*}
\sum_{0<\gamma}\frac{1}{\gamma^{2}}  &  <2.31050\times10^{-2},\\
\sum_{0<\gamma}\frac{1}{\gamma^{3}}  &  <7.29549\times10^{-4}.
\end{align*}

\end{lemma}

\begin{proof}
The first assertion is proved via Lemma \ref{backlundestimate} and the proof can also be found in (\cite{Lehman}, Lemma 2). For a proof of the second assertion we refer to (\cite{SaouterandDemichelandTrudgian}, Lemma 2.9)
\end{proof}

\begin{lemma}
\label{gammafintesumestimate} If $T\geq2\pi e$, then for some $\left\vert
\delta\right\vert \leq1$ we have%
\[
\sum_{0<\gamma\leq T}\frac{1}{\gamma}=\frac{1}{4\pi}\log^{2}\frac{T}{2\pi
}+0.9321\delta.
\]

\end{lemma}

\begin{proof}
The proof can be found in (\cite{SaouterandDemichelandTrudgian}, Lemma 2.10).
\end{proof}

\section{The Error terms}

\label{ErrorTerms}

As before, we denote by $A$ a positive number such that $\beta=1/2$ for all
zeros $\rho=\beta+i\gamma$ of $\zeta\left(  s\right)  $ for which
$0<\gamma\leq A$. For height greater than $A$ we only assume $0<\beta<1$ in
case the Riemann hypothesis is not true in general. We denote by%
\[
I\left(  \omega,\eta\right)  =\int_{\omega-\eta}^{\omega+\eta}K\left(
u-\omega\right)  ue^{-u/2}\left(  \pi\left(  e^{u}\right)  -li\left(
e^{u}\right)  \right)  du.
\]
Then, for $\omega-\eta\geq44.22$ we obtain from (\ref{ineq03}) and
(\ref{LiLehman}) together with $n=2$ the expression%
\begin{align}
I\left(  \omega,\eta\right)   &  \geq-\int_{\omega-\eta}^{\omega+\eta}K\left(
u-\omega\right)  du\nonumber\\
&  -\int_{\omega-\eta}^{\omega+\eta}K\left(  u-\omega\right)  \left(  \frac
{2}{u}+\frac{8}{u^{2}}+\frac{58.56}{u^{3}}+\log2ue^{-\frac{u}{2}}+\frac
{2}{\log2}ue^{-\frac{u}{6}}\right)  du\nonumber\\
&  -\sum_{\substack{\rho\\\left\vert \gamma\right\vert \leq A}}\frac{1}{\rho
}\int_{\omega-\eta}^{\omega+\eta}e^{u\left(  \rho-\frac{1}{2}\right)
}K\left(  u-\omega\right)  du\nonumber\\
&  -\sum_{\substack{\rho\\\left\vert \gamma\right\vert \leq A}}\frac{1}%
{\rho^{2}}\int_{\omega-\eta}^{\omega+\eta}\frac{1}{u}e^{u\left(  \rho-\frac
{1}{2}\right)  }K\left(  u-\omega\right)  du\label{Endformula}\\
&  -\sum_{\substack{\rho\\\left\vert \gamma\right\vert \leq A}}\int
_{\omega-\eta}^{\omega+\eta}2ue^{u\left(  \rho-\frac{1}{2}\right)  }K\left(
u-\omega\right)  \int_{0}^{\infty}\frac{e^{-t}}{\left(  u\rho-t\right)  ^{3}%
}dtdu\nonumber\\
&  -\sum_{\substack{\rho\\\left\vert \gamma\right\vert >A}}\int_{\omega-\eta
}^{\omega+\eta}K\left(  u-\omega\right)  ue^{-\frac{u}{2}}li\left(  e^{u\rho
}\right)  du,\nonumber
\end{align}
where the interchanging of summation and integration is justified by Arzela's
bounded convergence theorem because the series in (\ref{ineq03}) meets its
requirements on every relevant interval $\left[  \omega-\eta,\omega
+\eta\right]  $.

\subsection{The $R_{1}$ to $R_{5}$ terms}

We begin with the $R_{1}$ term.
\begin{align}
&  \int_{\omega-\eta}^{\omega+\eta}K\left(  u-\omega\right)  \left(  \frac
{2}{u}+\frac{8}{u^{2}}+\frac{58.56}{u^{3}}+\log2ue^{-\frac{u}{2}}+\frac
{2}{\log2}ue^{-\frac{u}{6}}\right)  du\nonumber\\
&  \leq\left(  \frac{2}{\omega-\eta}+\frac{8}{\left(  \omega-\eta\right)
^{2}}+\frac{58.56}{\left(  \omega-\eta\right)  ^{3}}+\log2\left(  \omega
+\eta\right)  e^{-\frac{\omega-\eta}{2}}+\frac{2}{\log2}\left(  \omega
+\eta\right)  e^{-\frac{\omega-\eta}{6}}\right) \nonumber\\
&  \cdot\int_{\omega-\eta}^{\omega+\eta}K\left(  u-\omega\right)
du\nonumber\\
&  \leq\frac{2}{\omega-\eta}+\frac{8}{\left(  \omega-\eta\right)  ^{2}}%
+\frac{58.56}{\left(  \omega-\eta\right)  ^{3}}+\log2\left(  \omega
+\eta\right)  e^{-\frac{\omega-\eta}{2}}+\frac{2}{\log2}\left(  \omega
+\eta\right)  e^{-\frac{\omega-\eta}{6}}\nonumber\\
&  =R_{1}. \label{R1term}%
\end{align}

\noindent Next, before turning to the $R_{2}$ term, we first study the $R_{3}
$ term which comes from formula%
\[
\sum_{\substack{\rho\\\left\vert \gamma\right\vert \leq A}}\frac{1}{\rho}%
\int_{\omega-\eta}^{\omega+\eta}e^{u\left(  \rho-\frac{1}{2}\right)  }K\left(
u-\omega\right)  du.
\]
In this situation we have $\rho=1/2+i\gamma$ and by using Lemma
\ref{Kxequation3} we get
\begin{align*}
&  \sum_{\substack{\rho\\\left\vert \gamma\right\vert \leq A}}\frac{1}{\rho
}\int_{\omega-\eta}^{\omega+\eta}e^{u\left(  \rho-\frac{1}{2}\right)
}K\left(  u-\omega\right)  du\\
&  =\sum_{\substack{\rho\\\left\vert \gamma\right\vert \leq A}}\frac{1}{\rho
}\int_{\omega-\eta}^{\omega+\eta}e^{iu\gamma}K\left(  u-\omega\right)  du\\
&  =\sum_{\substack{\rho\\\left\vert \gamma\right\vert \leq A}}\frac
{e^{i\omega\gamma}}{\rho}e^{-\frac{\gamma^{2}}{2\alpha}}-2\sum_{\substack{\rho
\\\left\vert \gamma\right\vert \leq A}}\frac{e^{i\omega\gamma}}{\rho}%
\int_{\eta}^{\infty}K\left(  x\right)  \operatorname{Re}\left(  e^{ix\gamma
}\right)  dx\\
&  \leq\sum_{\substack{\rho\\\left\vert \gamma\right\vert \leq A}%
}\frac{e^{i\omega\gamma}}{\rho}e^{-\frac{\gamma^{2}}{2\alpha}}+4\sum
_{\substack{\rho\\0<\gamma\leq A}}\left\vert \frac{e^{i\omega\gamma}}{\rho
}\right\vert \left\vert \int_{\eta}^{\infty}K\left(  x\right)
\operatorname{Re}\left(  e^{ix\gamma}\right)  dx\right\vert .
\end{align*}
Now, using Lemma \ref{Kxestimate3} together with Lemma \ref{gammaestimate}b we
get
\begin{align*}
&  \sum_{\substack{\rho\\\left\vert \gamma\right\vert \leq A}}\frac
{e^{i\omega\gamma}}{\rho}e^{-\frac{\gamma^{2}}{2\alpha}}+4\sum_{\substack{\rho
\\0<\gamma\leq A}}\left\vert \frac{e^{i\omega\gamma}}{\rho}\right\vert
\left\vert \int_{\eta}^{\infty}K\left(  x\right)  \operatorname{Re}\left(
e^{ix\gamma}\right)  dx\right\vert \\
&  \leq\sum_{\substack{\rho\\\left\vert \gamma\right\vert \leq A}%
}\frac{e^{i\omega\gamma}}{\rho}e^{-\frac{\gamma^{2}}{2\alpha}}+8K\left(
\eta\right)  \sum_{0<\gamma}\frac{1}{\gamma^{2}}\\
&  \leq\sum_{\substack{\rho\\\left\vert \gamma\right\vert \leq A}%
}\frac{e^{i\omega\gamma}}{\rho}e^{-\frac{\gamma^{2}}{2\alpha}}+0.074\sqrt
{\alpha}e^{-\frac{\alpha}{2}\eta^{2}}.
\end{align*}
We collect%
\begin{equation}
\sum_{\substack{\rho\\\left\vert \gamma\right\vert \leq A}}\frac{1}{\rho}%
\int_{\omega-\eta}^{\omega+\eta}e^{u\left(  \rho-\frac{1}{2}\right)  }K\left(
u-\omega\right)  du\leq\sum_{\substack{\rho\\\left\vert \gamma\right\vert \leq
A}}\frac{e^{i\omega\gamma}}{\rho}e^{-\frac{\gamma^{2}}{2\alpha}}+R_{3}.
\label{R3term}%
\end{equation}

\noindent Now, we proceed further with the $R_{2}$ term which comes from
\[
\sum_{\substack{\rho\\\left\vert \gamma\right\vert \leq A}}\frac{1}{\rho^{2}%
}\int_{\omega-\eta}^{\omega+\eta}\frac{1}{u}e^{u\left(  \rho-\frac{1}%
{2}\right)  }K\left(  u-\omega\right)  du.
\]
Again, we have $\rho=1/2+i\gamma$ and by using Lemma \ref{Kxequation3}
together with some standard arguments we calculate%
\begin{align*}
&  \sum_{\substack{\rho\\\left\vert \gamma\right\vert \leq A}}\frac{1}%
{\rho^{2}}\int_{\omega-\eta}^{\omega+\eta}\frac{1}{u}e^{u\left(  \rho-\frac
{1}{2}\right)  }K\left(  u-\omega\right)  du\\
&  =\sum_{\substack{\rho\\\left\vert \gamma\right\vert \leq A}}\frac
{e^{i\omega\gamma}}{\omega\rho^{2}}\int_{-\eta}^{\eta}\left(  1-\frac
{x}{\omega+x}\right)  K\left(  x\right)  e^{ix\gamma}dx\\
&  =\sum_{\substack{\rho\\\left\vert \gamma\right\vert \leq A}}\frac
{e^{i\omega\gamma}}{\omega\rho^{2}}e^{-\frac{\gamma^{2}}{2\alpha}}%
-2\sum_{\substack{\rho\\\left\vert \gamma\right\vert \leq A}}\frac
{e^{i\omega\gamma}}{\omega\rho^{2}}\int_{\eta}^{\infty}K\left(  x\right)
\operatorname{Re}\left(  e^{ix\gamma}\right)  dx\\
&  -\sum_{\substack{\rho\\\left\vert \gamma\right\vert \leq A}}\frac
{e^{i\omega\gamma}}{\omega\rho^{2}}\int_{-\eta}^{\eta}\frac{1}{\omega
+x}xK\left(  x\right)  e^{ix\gamma}dx\\
&  \leq\sum_{\substack{\rho\\\left\vert \gamma\right\vert \leq A}%
}\frac{e^{i\omega\gamma}}{\omega\rho^{2}}e^{-\frac{\gamma^{2}}{2\alpha}}%
+4\sum_{\substack{\rho\\0<\gamma\leq A}}\left\vert \frac{e^{i\omega\gamma}%
}{\omega\rho^{2}}\right\vert \left\vert \int_{\eta}^{\infty}K\left(  x\right)
e^{ix\gamma}dx\right\vert \\
&  +2\sum_{\substack{\rho\\0<\gamma\leq A}}\frac{1}{\omega\gamma^{2}}%
\int_{-\eta}^{\eta}\frac{1}{\omega+x}\left\vert x\right\vert K\left(
x\right)  dx.
\end{align*}
Together with Lemma \ref{Kxestimate3} and Lemma \ref{Kxequation2} we estimate
\begin{align*}
&  \sum_{\substack{\rho\\\left\vert \gamma\right\vert \leq A}}\frac{1}%
{\rho^{2}}\int_{\omega-\eta}^{\omega+\eta}\frac{1}{u}e^{u\left(  \rho-\frac
{1}{2}\right)  }K\left(  u-\omega\right)  du\\
&  \leq\sum_{\substack{\rho\\\left\vert \gamma\right\vert \leq A}%
}\frac{e^{i\omega\gamma}}{\omega\rho^{2}}e^{-\frac{\gamma^{2}}{2\alpha}}%
+\frac{\sqrt{\alpha}}{\omega}e^{-\frac{\alpha}{2}\eta^{2}}\sqrt{\frac{8}{\pi}%
}\min\left\{  2\sum_{0<\gamma}\frac{1}{\gamma^{3}},\frac{1}{\alpha\eta}%
\sum_{0<\gamma}\frac{1}{\gamma^{2}}\right\} \\
&  +\sqrt{\frac{8}{\pi}}\sum_{0<\gamma}\frac{1}{\gamma^{2}}\frac{1}{\omega
}\frac{1}{\omega-\eta}\frac{1}{\sqrt{\alpha}}\left(  1-e^{-\frac{\alpha}%
{2}\eta^{2}}\right)  .
\end{align*}
Collecting the numerical values in Lemma \ref{gammaestimate}b we finally
arrive at%
\begin{align}
&  \sum_{\substack{\rho\\\left\vert \gamma\right\vert \leq A}}\frac{1}%
{\rho^{2}}\int_{\omega-\eta}^{\omega+\eta}\frac{1}{u}e^{u\left(  \rho-\frac
{1}{2}\right)  }K\left(  u-\omega\right)  du\nonumber\\
&  \leq\sum_{\substack{\rho\\\left\vert \gamma\right\vert \leq A}%
}\frac{e^{i\omega\gamma}}{\omega\rho^{2}}e^{-\frac{\gamma^{2}}{2\alpha}}%
+\frac{0.037}{\omega\sqrt{\alpha}}\left(  \min\left\{  0.082\alpha,\frac
{1}{\eta}\right\}  e^{-\frac{\alpha}{2}\eta^{2}}+\frac{1-e^{-\frac{\alpha}%
{2}\eta^{2}}}{\omega-\eta}\right) \nonumber\\
&  =\sum_{\substack{\rho\\\left\vert \gamma\right\vert \leq A}}\frac
{e^{i\omega\gamma}}{\omega\rho^{2}}e^{-\frac{\gamma^{2}}{2\alpha}}+R_{2}.
\label{R2term}%
\end{align}

\noindent We proceed further with the $R_{4}$ term which comes from an
additional truncation at height $T$ below the possible maximum at height $A$
in the summation for $I\left(  \omega,\eta\right)  $ in Theorem
\ref{theorem31}. We collect the sums from (\ref{R3term}) and (\ref{R2term})
and we take the summation only up to height $T$, that is
\[
\sum_{\substack{\rho\\0<\left\vert \gamma\right\vert \leq T}}e^{i\omega\gamma
}\left(  \frac{1}{\rho}+\frac{1}{\omega\rho^{2}}\right)  e^{-\gamma
^{2}/2\alpha},
\]
for which we have to accept a further additional error term. Thus we split the
overall sum
\[
\sum_{\substack{\rho\\0<\left\vert \gamma\right\vert \leq A}}=\sum
_{\substack{\rho\\0<\left\vert \gamma\right\vert \leq T}}+\sum_{\substack{\rho
\\T<\left\vert \gamma\right\vert \leq A}}.
\]
Then by some simple manipulations we estimate
\begin{align*}
&  \sum_{\substack{\rho\\T<\left\vert \gamma\right\vert \leq A}}e^{i\omega
\gamma}\left(  \frac{1}{\rho}+\frac{1}{\omega\rho^{2}}\right)  e^{-\gamma
^{2}/2\alpha}\\
&  \leq2\sum_{T<\gamma\leq A}\left(  \frac{1}{\gamma}+\frac{1}{\omega
\gamma^{2}}\right)  e^{-\gamma^{2}/2\alpha}.
\end{align*}
Now, for $2\pi e\leq T$ we use Lemma \ref{backlundestimate} together with the
function $f\left(  x\right)  =\left(  \frac{2}{x}+\frac{2}{wx^{2}}\right)
e^{-x^{2}/2\alpha}$ from which we deduce that for some $\left\vert
\delta\right\vert \leq1$,
\begin{align*}
&  2\sum_{T<\gamma\leq A}\left(  \frac{1}{\gamma}+\frac{1}{\omega\gamma^{2}%
}\right)  e^{-\gamma^{2}/2\alpha}\\
&  \leq\int_{T}^{\infty}\left(  \frac{1}{\pi x}+\frac{1}{\pi\omega x^{2}%
}\right)  e^{-\frac{x^{2}}{2\alpha}}\log\frac{x}{2\pi}dx\\
&  +\delta\left(  \left(  \frac{8}{T}+\frac{8}{\omega T^{2}}\right)
e^{-\frac{T^{2}}{2\alpha}}\log T+4\int_{T}^{\infty}\left(  \frac{1}{x^{2}%
}+\frac{1}{\omega x^{3}}\right)  e^{-\frac{x^{2}}{2\alpha}}dx\right)  .
\end{align*}
Now we set
\begin{align*}
h_{1}\left(  x\right)   &  =\left(  \frac{1}{\pi x}+\frac{1}{\pi\omega x^{2}%
}\right)  \log\frac{x}{2\pi},\\
h_{2}\left(  x\right)   &  =\frac{1}{x^{2}}+\frac{1}{\omega x^{3}}%
\end{align*}
then, by combining together with Lemma \ref{Kxestimate2}, we finally arrive
at
\begin{align*}
&  2\sum_{T<\gamma\leq A}\left(  \frac{1}{\gamma}+\frac{1}{\omega\gamma^{2}%
}\right)  e^{-\gamma^{2}/2\alpha}\\
&  \leq e^{-\frac{T^{2}}{2\alpha}}\left(  \frac{\alpha}{\pi T^{2}}\log\frac
{T}{2\pi}+\frac{8\log T}{T}+\frac{4\alpha}{T^{3}}\right)  \left(  1+\frac
{1}{\omega T}\right) \\
&  =R_{4}.
\end{align*}
We collect
\begin{align}
&  \sum_{\substack{\rho\\0<\left\vert \gamma\right\vert \leq A}}e^{i\omega
\gamma}\left(  \frac{1}{\rho}+\frac{1}{\omega\rho^{2}}\right)  e^{-\gamma
^{2}/2\alpha}\nonumber\\
&  \leq\sum_{\substack{\rho\\0<\left\vert \gamma\right\vert \leq T}%
}e^{i\omega\gamma}\left(  \frac{1}{\rho}+\frac{1}{\omega\rho^{2}}\right)
e^{-\gamma^{2}/2\alpha}+R_{4}. \label{R4term}%
\end{align}
The estimate for the $R_{5}$ term is done by some simple computations. For the
last time we can take use of $\rho=1/2+i\gamma$ and by the triangle inequality
we estimate%
\begin{align}
&  \sum_{\substack{\rho\\\left\vert \gamma\right\vert \leq A}}\int
_{\omega-\eta}^{\omega+\eta}2ue^{u\left(  \rho-\frac{1}{2}\right)  }K\left(
u-\omega\right)  \int_{0}^{\infty}\frac{e^{-t}}{\left(  u\rho-t\right)  ^{3}%
}dtdu\nonumber\\
&  =\sum_{\substack{\rho\\\left\vert \gamma\right\vert \leq A}}\int
_{\omega-\eta}^{\omega+\eta}2ue^{iu\gamma}K\left(  u-\omega\right)  \int
_{0}^{\infty}\frac{e^{-t}}{\left(  u\rho-t\right)  ^{3}}dtdu\nonumber\\
&  \leq\sum_{\substack{\rho\\\left\vert \gamma\right\vert \leq A}}\int
_{\omega-\eta}^{\omega+\eta}2uK\left(  u-\omega\right)  \int_{0}^{\infty}%
\frac{e^{-t}}{\left\vert u\rho-t\right\vert ^{3}}dtdu\nonumber\\
&  \leq\sum_{\substack{\rho\\\left\vert \gamma\right\vert \leq A}}\int
_{\omega-\eta}^{\omega+\eta}2uK\left(  u-\omega\right)  \int_{0}^{\infty}%
\frac{e^{-t}}{u^{3}\left\vert \gamma\right\vert ^{3}}dtdu\nonumber\\
&  \leq\frac{4}{\left(  \omega-\eta\right)  ^{2}}\sum_{0<\gamma}\frac
{1}{\gamma^{3}}\int_{-\eta}^{\eta}K\left(  x\right)  dx\nonumber\\
&  \leq\frac{0.003}{\left(  \omega-\eta\right)  ^{2}}\nonumber\\
&  =R_{5}, \label{R5term}%
\end{align}
where we have again used Lemma \ref{gammaestimate}b.

\subsection{The final $R_{6}$ term}

The major part consists now in a careful estimate for the last remaining term
from (\ref{Endformula}), that is
\[
\sum_{\substack{\rho\\\left\vert \gamma\right\vert >A}}\int_{\omega-\eta
}^{\omega+\eta}K\left(  u-\omega\right)  ue^{-\frac{u}{2}}li\left(  e^{u\rho
}\right)  du.
\]
The proof for the estimate starts with a slightly different setting on the
analytic region, but follows Lehman's principal idea. Later on we deviate from
Lehman to improve the estimates. We introduce the sector region
\[
D=\left\{  z\in\mathbb{C}:-\frac{3}{10}\pi\leq\arg\left(  z\right)  \leq
\frac{3}{10}\pi\right\}
\]
together with the function $f:D\rightarrow\mathbb{C}$ by defining
\[
f_{\rho}\left(  z\right)  =\rho ze^{-\rho z}li\left(  e^{\rho z}\right)
e^{-\frac{\alpha}{2}\left(  z-\omega\right)  ^{2}}.
\]
We may write%
\begin{align*}
&  \int_{\omega-\eta}^{\omega+\eta}K\left(  u-\omega\right)  ue^{-\frac{u}{2}%
}li\left(  e^{u\rho}\right)  du\\
&  =\sqrt{\frac{\alpha}{2\pi}}\frac{1}{\rho}\int_{\omega-\eta}^{\omega+\eta
}e^{u\left(  \rho-\frac{1}{2}\right)  }f_{\rho}\left(  u\right)  du.
\end{align*}
Now, following Lehman's idea, we perform successive partial integration and,
for a positive integer $N$ which we fix later, we get
\begin{align*}
&  \int_{\omega-\eta}^{\omega+\eta}e^{u\left(  \rho-\frac{1}{2}\right)
}f_{\rho}\left(  u\right)  du\\
&  =\sum_{n=0}^{N-1}\left(  -1\right)  ^{n}\frac{e^{\omega\left(  \rho
-\frac{1}{2}\right)  }}{\left(  \rho-\frac{1}{2}\right)  ^{n+1}}\left(
e^{\eta\left(  \rho-\frac{1}{2}\right)  }f_{\rho}^{\left(  n\right)  }\left(
\omega+\eta\right)  -e^{-\eta\left(  \rho-\frac{1}{2}\right)  }f_{\rho
}^{\left(  n\right)  }\left(  \omega-\eta\right)  \right) \\
&  +\frac{\left(  -1\right)  ^{N}}{\left(  \rho-\frac{1}{2}\right)  ^{N}}%
\int_{\omega-\eta}^{\omega+\eta}e^{u\left(  \rho-\frac{1}{2}\right)  }f_{\rho
}^{\left(  N\right)  }\left(  u\right)  du.
\end{align*}
From this we obtain for all $N\geq1$ the estimate
\begin{align}
&  \left\vert \int_{\omega-\eta}^{\omega+\eta}e^{u\left(  \rho-\frac{1}%
{2}\right)  }f_{\rho}\left(  u\right)  du\right\vert \nonumber\\
&  \leq e^{\frac{\omega+\eta}{2}}\sum_{n=0}^{N-1}\frac{\left\vert f_{\rho
}^{\left(  n\right)  }\left(  \omega+\eta\right)  \right\vert +\left\vert
f_{\rho}^{\left(  n\right)  }\left(  \omega-\eta\right)  \right\vert
}{\left\vert \gamma\right\vert ^{n+1}}\nonumber\\
&  +e^{\frac{\omega+\eta}{2}}\frac{1}{\left\vert \gamma\right\vert ^{N}}%
\int_{\omega-\eta}^{\omega+\eta}\left\vert f_{\rho}^{\left(  N\right)
}\left(  u\right)  \right\vert du. \label{Partialintegrationestimate}%
\end{align}
For all $z\in D$ and further, since $\left\vert \operatorname{Im}\left(
\rho\right)  \right\vert >14.1$ for all non trivial zeta zeros, we have%
\begin{align*}
0.4774\pi &  <\left\vert \arg\left(  \rho\right)  \right\vert <\frac{\pi}%
{2},\\
0  &  \leq\left\vert \arg\left(  z\right)  \right\vert \leq\frac{3}{10}\pi.
\end{align*}
Thus%
\[
\frac{\pi}{6}<0.1774\pi\leq\left\vert \arg\left(  \rho z\right)  \right\vert
\leq\frac{4}{5}\pi<\frac{5}{6}\pi,
\]
from which, by a previous remark, we obtain that $f_{\rho}$ is an analytic
function in the sector plane $D$. Now let%
\begin{align}
0  &  <\eta<\frac{\omega}{100},\label{firstrestriction}\\
\frac{5A}{4\omega}  &  \leq\alpha\leq A^{2}. \label{secondrestriction}%
\end{align}
If we select a value $x\in\left[  \omega-\eta,\omega+\eta\right]  $ and take a
circle contour $c\left(  t\right)  =x+re^{it}$ for $t\in\left[  0,2\pi\right]
$ then with a radius $r\leq\frac{4}{5}\omega$ we stay always inside the
analytic region $D$, since by (\ref{firstrestriction}) we have
\begin{align*}
r  &  \leq\frac{4}{5}\omega<\omega\left(  1-\frac{1}{100}\right)  \sin\frac
{3}{10}\pi\\
&  \leq\left(  \omega-\eta\right)  \sin\frac{3}{10}\pi,
\end{align*}
where the latter quantity represents the radius of the touching circle
centered at $\omega-\eta$ with the sector border line. We collect
\[
\left\vert \arg\left(  c\left(  t\right)  \right)  \right\vert <\frac{3}%
{10}\pi,\forall t\in\left[  0,2\pi\right]  .
\]
for all circle contours, centered at $x\in\left[  \omega-\eta,\omega
+\eta\right]  $ with radius $r$ not larger than $\frac{4}{5}\omega$. \ Later
we will take use of those contours.

\noindent For $z\in D$, using (\ref{LiLehman}) with $n=1$ and together with
the observation that $\left\vert \operatorname{Im}\left(  \rho z\right)
\right\vert =\left\vert \rho z\right\vert \left\vert \sin\left(  \arg\left(
\rho z\right)  \right)  \right\vert \geq\left\vert \rho z\right\vert
\left\vert \sin\frac{\pi}{6}\right\vert =\frac{1}{2}\left\vert \rho
z\right\vert ,$ we find the estimate
\begin{equation}
\left\vert f_{\rho}\left(  z\right)  \right\vert \leq\left(  1+\frac
{4}{\left\vert \gamma\right\vert \left\vert z\right\vert }\right)  \left\vert
e^{-\frac{\alpha}{2}\left(  z-\omega\right)  ^{2}}\right\vert .
\label{frhoestimate1}%
\end{equation}
We now follow Lehman by estimating $f_{\rho}^{\left(  n\right)  }$ with \ a
contour integral
\[
f_{\rho}^{\left(  n\right)  }\left(  x\right)  =\frac{n!}{2\pi i}%
%TCIMACRO{\doint }%
%BeginExpansion
{\displaystyle\oint}
%EndExpansion
\frac{f_{\rho}\left(  z\right)  }{\left(  z-x\right)  ^{n+1}}dz,
\]
where we take the contour $c\left(  t\right)  =x+re^{it}.$ Proceeding in a
standard pattern and using (\ref{frhoestimate1}) we further estimate
\begin{align}
\left\vert f_{\rho}^{\left(  n\right)  }\left(  x\right)  \right\vert  &
\leq\frac{n!}{2\pi}\frac{1}{r^{n}}\int_{0}^{2\pi}\frac{\left\vert f_{\rho
}\left(  c\left(  t\right)  \right)  \right\vert }{1}dt\nonumber\\
&  \leq\frac{n!}{r^{n}}\max_{t\in\left[  0,2\pi\right]  }\left\vert f_{\rho
}\left(  x+re^{it}\right)  \right\vert \nonumber\\
&  \leq\left(  1+\frac{4}{\left\vert \gamma\right\vert \left\vert
x-r\right\vert }\right)  \frac{n!}{r^{n}}\max_{t\in\left[  0,2\pi\right]
}\left\vert e^{-\frac{\alpha}{2}\left(  x-\omega+re^{it}\right)  ^{2}%
}\right\vert . \label{frhoestimate2}%
\end{align}
Now, we take a look on the $\max$-term in (\ref{frhoestimate2}). Some routine
arguments show that%
\begin{align}
&  \max_{t\in\left[  0,2\pi\right]  }\left\vert e^{-\frac{\alpha}{2}\left(
x-\omega+re^{it}\right)  ^{2}}\right\vert \nonumber\\
&  =e^{-\frac{\alpha}{2}\left(  x-\omega\right)  ^{2}}e^{\frac{\alpha}{2}%
r^{2}}\max_{t\in\left[  0,2\pi\right]  }e^{\alpha r\left(  \left(  w-x\right)
\cos t-r\cos^{2}t\right)  }. \label{frhoestimate3}%
\end{align}
By some standard analysis arguments it is now easy to locate the maximum
values of this latter expression. For all $r>0$ we find%
\begin{equation}
\max_{t\in\left[  0,2\pi\right]  }e^{\alpha r\left(  \left(  w-x\right)  \cos
t-r\cos^{2}t\right)  }\leq e^{\alpha r\frac{\left(  \omega-x\right)  ^{2}}%
{4r}}=e^{\frac{\alpha}{4}\left(  \omega-x\right)  ^{2}}. \label{frhoestimate4}%
\end{equation}
Now, combining (\ref{frhoestimate2}), (\ref{frhoestimate3}) together with
(\ref{frhoestimate4}) we conclude that%
\[
\left\vert f_{\rho}^{\left(  n\right)  }\left(  x\right)  \right\vert
\leq\left(  1+\frac{4}{\left\vert \gamma\right\vert \left\vert x-r\right\vert
}\right)  \frac{n!}{r^{n}}e^{\frac{\alpha}{2}r^{2}}e^{-\frac{\alpha}{4}\left(
x-\omega\right)  ^{2}}.
\]
Together with (\ref{firstrestriction}) and $\left\vert x-r\right\vert \geq
x-r\geq\omega-\eta-r>\omega-\frac{\omega}{100}-\frac{4}{5}\omega=\frac
{19}{100}\omega$ and $\left\vert \operatorname{Im}\left(  \rho\right)
\right\vert =\left\vert \gamma\right\vert >A$ we arrive at%
\begin{equation}
\left\vert f_{\rho}^{\left(  n\right)  }\left(  x\right)  \right\vert
\leq\left(  1+\frac{22}{A\omega}\right)  \frac{n!}{r^{n}}e^{\frac{\alpha}%
{2}r^{2}}e^{-\frac{\alpha}{4}\left(  x-\omega\right)  ^{2}},
\label{frhoestimate5}%
\end{equation}
which is valid for $n\in\mathbb{N},0<r\leq\frac{4}{5}\omega,x\in\left[
w-\eta,\omega+\eta\right]  $ and $0<\eta<\frac{w}{100}.$ We fix now $N=\left[
A^{2}/\alpha\right]  $, which by (\ref{secondrestriction}) is a valid natural
number. Now we deviate from Lehman. For each individual $1\leq n\leq N$ we
select now $r=\sqrt{n/\alpha}$, which, again by (\ref{secondrestriction}) is a
valid choice for our purposes, since%
\[
r=\sqrt{\frac{n}{\alpha}}\leq\sqrt{\frac{N}{\alpha}}=\sqrt{\frac{\left[
\frac{A^{2}}{\alpha}\right]  }{\alpha}}\leq\frac{A}{\alpha}\leq\frac{4}%
{5}\omega,
\]
and inserting into (\ref{frhoestimate5}) establishes%
\begin{equation}
\left\vert f_{\rho}^{\left(  n\right)  }\left(  x\right)  \right\vert
\leq\left(  1+\frac{22}{A\omega}\right)  n!\left(  \frac{\alpha e}{n}\right)
^{\frac{n}{2}}e^{-\frac{\alpha}{4}\left(  x-\omega\right)  ^{2}}.
\label{frhoestimate6}%
\end{equation}
Armed with this latter estimate we start now estimating
(\ref{Partialintegrationestimate}) by finding relevant expressions for
$\left\vert f_{\rho}^{\left(  n\right)  }\left(  \omega\pm\eta\right)
\right\vert $ for $0\leq n\leq N-1$ and $\left\vert f_{\rho}^{\left(
N\right)  }\left(  x\right)  \right\vert $ with $x\in\left[  w-\eta
,\omega+\eta\right]  $. By combining (\ref{frhoestimate1}) together with
(\ref{frhoestimate6}), we finally arrive at%
\begin{align}
\left\vert f_{\rho}^{\left(  n\right)  }\left(  \omega\pm\eta\right)
\right\vert  &  \leq\left\{
\begin{array}
[c]{ll}%
\left(  1+\frac{22}{A\omega}\right)  e^{-\frac{\alpha}{4}\eta^{2}}, & n=0,\\
\left(  1+\frac{22}{A\omega}\right)  n!\left(  \frac{\alpha e}{n}\right)
^{\frac{n}{2}}e^{-\frac{\alpha}{4}\eta^{2}}, & n=1,2,\ldots,N-1,
\end{array}
\right. \nonumber\\
\left\vert f_{\rho}^{\left(  N\right)  }\left(  x\right)  \right\vert  &
\leq\left(  1+\frac{22}{A\omega}\right)  N!\left(  \frac{\alpha e}{N}\right)
^{\frac{N}{2}}. \label{frhoestimate7}%
\end{align}
Collecting now (\ref{Partialintegrationestimate}) together with
(\ref{frhoestimate7}) we establish%
\begin{align}
&  \left\vert \int_{\omega-\eta}^{\omega+\eta}e^{u\left(  \rho-\frac{1}%
{2}\right)  }f_{\rho}\left(  u\right)  du\right\vert \nonumber\\
&  \leq\frac{2\left(  1+\frac{22}{A\omega}\right)  }{\left\vert \gamma
\right\vert }e^{\frac{\omega+\eta}{2}}\left(  e^{-\frac{\alpha}{4}\eta^{2}%
}+\sum_{n=1}^{N-1}n!\left(  \frac{\alpha e}{n}\right)  ^{\frac{n}{2}}%
e^{-\frac{\alpha}{4}\eta^{2}}\frac{1}{\left\vert \gamma\right\vert ^{n}%
}\right. \nonumber\\
&  +\left.  \frac{N!\eta\left(  \frac{\alpha e}{N}\right)  ^{\frac{N}{2}}%
}{\left\vert \gamma\right\vert ^{N-1}}\right)  . \label{Sumestimate1}%
\end{align}
With $n!\leq\left(  \frac{n}{e}\right)  ^{n}e\sqrt{n},n\geq1$ we calculate%
\begin{align}
\sum_{n=1}^{N-1}n!\left(  \frac{\alpha e}{n}\right)  ^{\frac{n}{2}}\frac
{1}{\left\vert \gamma\right\vert ^{n}}  &  \leq\sum_{n=1}^{N-1}e\sqrt
{n}\left(  \frac{\alpha n}{e\gamma^{2}}\right)  ^{\frac{n}{2}}\nonumber\\
&  \leq\frac{eA}{\sqrt{\alpha}}\sum_{n=1}^{\infty}\left(  \frac{1}{\sqrt{e}%
}\right)  ^{n}\nonumber\\
&  =\frac{eA}{\sqrt{\alpha}}\frac{1}{\sqrt{e}-1}. \label{Sumestimate2}%
\end{align}
Similar we operate with%
\begin{align}
\frac{N!\eta\left(  \frac{\alpha e}{N}\right)  ^{\frac{N}{2}}}{\left\vert
\gamma\right\vert ^{N-1}}  &  \leq Ae\eta\sqrt{N}e^{-\frac{N}{2}}\left(
\frac{N\alpha}{A^{2}}\right)  ^{\frac{N}{2}}\nonumber\\
&  \leq A^{2}e^{\frac{3}{2}}\eta\frac{1}{\sqrt{\alpha}}e^{-\frac{A^{2}%
}{2\alpha}}. \label{Sumestimate3}%
\end{align}
Again, collecting (\ref{Sumestimate1}), (\ref{Sumestimate2}) together with
(\ref{Sumestimate3}) we calculate%
\begin{align*}
&  \left\vert \int_{\omega-\eta}^{\omega+\eta}e^{u\left(  \rho-\frac{1}%
{2}\right)  }f_{\rho}\left(  u\right)  du\right\vert \\
&  \leq\frac{2e\sqrt{e}\left(  1+\frac{22}{Aw}\right)  }{\left\vert
\gamma\right\vert }\frac{A}{\sqrt{\alpha}}e^{\frac{\omega+\eta}{2}}\left(
e^{-\frac{\alpha}{4}\eta^{2}}\left(  \frac{\sqrt{\alpha}}{A}\frac{1}{e\sqrt
{e}}+\frac{1}{e-\sqrt{e}}\right)  +A\eta e^{-\frac{A^{2}}{2\alpha}}\right)  .
\end{align*}
From this estimate, using Lemma \ref{gammaestimate}a, we obtain our final
result by%
\begin{align}
&  \sum_{\substack{\rho\\\left\vert \gamma\right\vert >A}}\int_{\omega-\eta
}^{\omega+\eta}K\left(  u-\omega\right)  ue^{-\frac{u}{2}}li\left(  e^{u\rho
}\right)  du\nonumber\\
&  \leq\sqrt{\frac{\alpha}{2\pi}}\sum_{\substack{\rho\\\left\vert
\gamma\right\vert >A}}\frac{1}{\left\vert \gamma\right\vert }\left\vert
\int_{\omega-\eta}^{\omega+\eta}e^{u\left(  \rho-\frac{1}{2}\right)  }f_{\rho
}\left(  u\right)  du\right\vert \nonumber\\
&  \leq\frac{4}{\sqrt{2\pi}}\left(  1+\frac{22}{Aw}\right)  A\log A\left(
\left(  1+\frac{e}{\sqrt{e}-1}\right)  \frac{1}{A}e^{-\frac{\alpha}{4}\eta
^{2}+\frac{\omega+\eta}{2}}+e\sqrt{e}\eta e^{-\frac{A^{2}}{2\alpha}%
+\frac{\omega+\eta}{2}}\right) \nonumber\\
&  \leq\left(  1+\frac{22}{Aw}\right)  A\log A\left(  \frac{8.283}{A}%
e^{-\frac{\alpha}{4}\eta^{2}+\frac{\omega+\eta}{2}}+7.152\eta e^{-\frac{A^{2}%
}{2\alpha}+\frac{\omega+\eta}{2}}\right) \nonumber\\
&  =R_{6}. \label{R6term}%
\end{align}
Finally, we collect the results from (\ref{Endformula}), (\ref{R1term}),
(\ref{R3term}), (\ref{R2term}), (\ref{R4term}), (\ref{R5term}), (\ref{R6term})
and we arrive at Theorem \ref{theorem31}. In case of the Riemann hypothesis,
we can combine (\ref{ineq03}), (\ref{LiLehman}), (\ref{Endformula}) with
$A\rightarrow\infty$, and together with (\ref{R1term}), (\ref{R3term}),
(\ref{R2term}), (\ref{R4term}), (\ref{R5term}) we obtain the conclusion of
Theorem \ref{theorem31} with the $R_{6}$ term in the estimate and the
conditions (\ref{Condition3}), (\ref{Condition4}) to be omitted.

\section{Numerical applications}

\label{Numerical applications}

We will give some demonstrations of how we can use Theorem \ref{theorem31} in
numerical practice to get sharper regions for the crossovers. First of all, it
is clear that the free choice of the value $\eta$ is in reality of limited
value for numerical applications, since this size $\eta$ seriously affects
some of the error terms. The most problematic error terms in this context are
the $R_{3}$ and $R_{6}$ terms. On the other hand, with increasing values
$\alpha$ the $R_{4}$ term becomes problematic, since we have to chose higher
values $T$ and the amount of numerical calculations in the sum for $I\left(
\omega,\eta\right)  $ is also raising up. Looking to the summation%
\[
\sum_{\substack{\rho\\0<\left\vert \gamma\right\vert \leq T}}e^{i\omega\gamma
}\left(  \frac{1}{\rho}+\frac{1}{\omega\rho^{2}}\right)  e^{-\gamma
^{2}/2\alpha}=S_{1}\left(  \alpha,\omega,T\right)  +S_{2}\left(  \alpha
,\omega,T\right)
\]
one can easily check, that the second term $1/\omega\rho^{2}$ contributes only
minor to the overall sum value, since by Lemma \ref{gammaestimate}b we have
\begin{equation}
\left\vert S_{2}\left(  \alpha,\omega,T\right)  \right\vert =\left\vert
\sum_{\substack{\rho\\0<\left\vert \gamma\right\vert \leq T}}\frac
{e^{i\omega\gamma}}{\omega\rho^{2}}e^{-\gamma^{2}/2\alpha}\right\vert
\leq\frac{2}{\omega}\sum_{\substack{\rho\\0<\gamma}}\frac{1}{\gamma^{2}}%
\leq\frac{1}{21\omega}. \label{Sumestimate4}%
\end{equation}
The real advantage of the $S_{2}$ summation is established in the better error
terms. The numerical summation is affected by the numerical accuracy of the
zeta zeros. Thus, following the error analysis by te Riele, we have to handle
two further additional error terms $\Delta S_{1}$ and $\Delta S_{2}$.

\subsection{The numerical sums}

We begin with the error analysis for the additional terms. With $\rho
=1/2+i\gamma$ we have%
\begin{align*}
\operatorname{Re}\left(  \frac{e^{i\omega\gamma}}{\rho}\right)   &  =\frac
{1}{\frac{1}{4}+\gamma^{2}}\left(  \frac{1}{2}\cos\omega\gamma+\gamma
\sin\omega\gamma\right)  ,\\
\operatorname{Re}\left(  \frac{e^{i\omega\gamma}}{\rho^{2}}\right)   &
=\frac{1}{\left(  \frac{1}{4}+\gamma^{2}\right)  ^{2}}\left(  \left(  \frac
{1}{4}-\gamma^{2}\right)  \cos\omega\gamma+\gamma\sin\omega\gamma\right)  .
\end{align*}
We set%
\begin{align*}
s\left(  \alpha,\omega,\gamma\right)   &  =\frac{1}{\frac{1}{4}+\gamma^{2}%
}\left(  \cos\omega\gamma+2\gamma\sin\omega\gamma\right)  e^{-\frac{\gamma
^{2}}{2\alpha}},\\
t\left(  \alpha,\omega,\gamma\right)   &  =\frac{1}{\omega\left(  \frac{1}%
{4}+\gamma^{2}\right)  ^{2}}\left(  \left(  \frac{1}{2}-2\gamma^{2}\right)
\cos\omega\gamma+2\gamma\sin\omega\gamma\right)  e^{-\frac{\gamma^{2}}%
{2\alpha}}.
\end{align*}
Then, we can write%
\begin{align*}
S_{1}\left(  \alpha,\omega,T\right)   &  =\sum_{\substack{\rho\\0<\left\vert
\gamma\right\vert \leq T}}\frac{e^{i\omega\gamma}}{\rho}e^{-\frac{\gamma^{2}%
}{2\alpha}}=\sum_{0<\gamma\leq T}s\left(  \alpha,\omega,\gamma\right)  ,\\
S_{2}\left(  \alpha,\omega,T\right)   &  =\sum_{\substack{\rho\\0<\left\vert
\gamma\right\vert \leq T}}\frac{e^{i\omega\gamma}}{\omega\rho^{2}}%
e^{-\frac{\gamma^{2}}{2\alpha}}=\sum_{0<\gamma\leq T}t\left(  \alpha
,\omega,\gamma\right)  .
\end{align*}
Now, we denote by $\gamma^{\ast}$ the numerical approximation for $\gamma$ for
which we can assume an overall bound\footnote{All numerical computations were
carried out on different Intel(R) Core(TM) i7-8700 CPU 3.20GHz platforms
simultaneously together with Mathematica 13.3 and Mathematica 14.1 software
packages.} by $\left\vert \gamma^{\ast}-\gamma\right\vert \leq10^{-9}$ for all
$0<\gamma\leq T$. The corresponding numerical sums are denoted by%
\begin{align*}
S_{1}^{\ast}\left(  \alpha,\omega,T\right)   &  =\sum_{0<\gamma\leq T}s\left(
\alpha,\omega,\gamma^{\ast}\right)  ,\\
S_{2}^{\ast}\left(  \alpha,\omega,T\right)   &  =\sum_{0<\gamma\leq T}t\left(
\alpha,\omega,\gamma^{\ast}\right)  .
\end{align*}
Further, denote by%
\[
S^{\ast}\left(  \alpha,\omega,T\right)  =S_{1}^{\ast}\left(  \alpha
,\omega,T\right)  +S_{2}^{\ast}\left(  \alpha,\omega,T\right)  .
\]
Then, by the mean value theorem, for some $x_{\gamma}\in\left(  \gamma^{\ast
},\gamma\right)  $ we calculate
\begin{align}
\Delta S_{1}  &  =\left\vert S_{1}^{\ast}\left(  \alpha,\omega,T\right)
-S_{1}\left(  \alpha,\omega,T\right)  \right\vert \nonumber\\
&  \leq\sum_{0<\gamma\leq T}\left\vert \frac{\partial}{\partial x}s\left(
\alpha,\omega,x_{\gamma}\right)  \right\vert \left\vert \gamma^{\ast}%
-\gamma\right\vert \nonumber\\
&  \leq10^{-9}\sum_{0<\gamma\leq T}\left\vert \frac{\partial}{\partial
x}s\left(  \alpha,\omega,x_{\gamma}\right)  \right\vert , \label{DeltaS1}%
\end{align}
with%
\begin{equation}
\left\vert \frac{\partial}{\partial x}s\left(  \alpha,\omega,x\right)
\right\vert \leq\frac{1}{x}\left(  2\omega+\frac{\omega}{x}+\frac{2x}{\alpha
}+\frac{2}{x^{2}}+\frac{4}{x}\right)  . \label{RieleS1}%
\end{equation}
The same calculation is done for the $S_{2}$ summation. For some $x_{\gamma
}\in\left(  \gamma^{\ast},\gamma\right)  $ we have
\begin{align}
\Delta S_{2}  &  =\left\vert S_{2}^{\ast}\left(  \alpha,\omega,T\right)
-S_{2}\left(  \alpha,\omega,T\right)  \right\vert \nonumber\\
&  \leq\sum_{0<\gamma\leq T}\left\vert \frac{\partial}{\partial x}t\left(
\alpha,\omega,x_{\gamma}\right)  \right\vert \left\vert \gamma^{\ast}%
-\gamma\right\vert \nonumber\\
&  \leq10^{-9}\sum_{0<\gamma\leq T}\left\vert \frac{\partial}{\partial
x}t\left(  \alpha,\omega,x_{\gamma}\right)  \right\vert , \label{DeltaS2}%
\end{align}
with%
\begin{equation}
\left\vert \frac{\partial}{\partial x}t\left(  \alpha,\omega,x\right)
\right\vert \leq\frac{1}{x^{2}}\left(  \frac{2}{x}+\frac{2x}{\alpha\omega
}+\frac{1}{2x^{2}}+2+\frac{8}{\omega x}+\frac{8}{\omega x^{2}}\right)  .
\label{RieleS2}%
\end{equation}

\medskip\noindent Now, we demonstrate the efficiency of our formulas by
applying them to the reworked Bays-Hudson region by Chao-Plymen (2010) and the
Saouter-Demichel region (2010) which lie on the front end of the lowest known
certified positive regions. We proceed by inserting the original values in our
error terms and then we resize $\eta$ so small as we can do with regard to the
resulting numerical overall error.

\subsection{Reworked Bays-Hudson region by Chao-Plymen, 2010}

\noindent Fixing $\omega<728,\alpha=1.34\cdot10^{11},T=\gamma_{2000000}%
<1131944.4718$ and $14<\gamma_{1}\leq\gamma<1131945$ we get from
(\ref{RieleS1}) and (\ref{RieleS2})%
\begin{align*}
\left\vert \frac{\partial}{\partial x}s\left(  \alpha,\omega,x\right)
\right\vert  &  \leq\frac{1508.3}{x},\\
\left\vert \frac{\partial}{\partial x}t\left(  \alpha,\omega,x\right)
\right\vert  &  \leq\frac{2.14625}{x^{2}},
\end{align*}
and, using Lemma \ref{gammafintesumestimate}, Lemma \ref{gammaestimate}b
together with $\left\vert x_{\gamma}-\gamma\right\vert <\left\vert
\gamma^{\ast}-\gamma\right\vert \leq10^{-9}$, we estimate (\ref{DeltaS1}) and
(\ref{DeltaS2})
\begin{align*}
\Delta S_{1}  &  \leq10^{-9}\times1508.3\times1.0001\times\sum_{0<\gamma\leq
T}\frac{1}{\gamma}\leq1.89855\times10^{-5},\\
\Delta S_{2}  &  \leq10^{-9}\times2.14625\times1.0001^{2}\times\sum_{0<\gamma
}\frac{1}{\gamma^{2}}\leq4.9599\times10^{-11}.
\end{align*}
With $\omega=727.952018$, $\eta=0.00016$ and $A=1.022\cdot10^{7}$ our machine
computation for the $S^{\ast}$ sum gives $S^{\ast}\left(  \alpha
,\omega,T\right)  =-1.006553478788955$, a slightly different value than the
one obtained by Chao and Plymen with size $-1.006569$, due to our double
summation. We set $R=R_{1}+R_{2}+R_{3}+R_{4}+R_{5}+R_{6}$. Now, we present the
following tables, where the column with the lower estimates for $I\left(
\omega,\eta\right)  $ is calculated by $I\left(  \omega,\eta\right)
\geq-1-S^{\ast}\left(  \alpha,\omega,T\right)  -\Delta S_{1}-\Delta S_{2}-R$.

\begin{table}[H]
\begin{center}
{
\begin{tabular}
[c]{|r|c|c|}\hline
\rowcolor{orange} $\eta$ & $R$ & Lower estimate for $I\left(  \omega
,\eta\right)  $\\\hline
$0.0001600$ & $6.14384 \times10^{-3}$ & $0.000390651$\\
$0.0001400$ & $6.14369 \times10^{-3}$ & $0.000390805$\\
$0.0001200$ & $6.14353 \times10^{-3}$ & $0.000390959$\\
$0.0001063$ & $6.20670 \times10^{-3}$ & $0.000327792$\\
$0.0001061$ & $6.40600 \times10^{-3}$ & $0.000128492$\\
$0.0001060$ & $6.67779 \times10^{-3}$ & negative\\
$0.0001050$ & $6.33687 \times10^{-1}$ & negative\\\hline
\end{tabular}
}
\end{center}
\caption{Resizing $\eta$ in the reworked Bays-Hudson region.}%
\label{ResizeBays}
\end{table}

\noindent Comparing to the original estimate for $I\left(  \omega,\eta\right)
$ by Chao and Plymen with calculated value $0.0002$ we perform better.
Moreover, the best resized value $\eta=0.0001061$ is improved by about $33.68$
percent compared to the original $\eta$.

\subsection{The Saouter-Demichel region, 2010}

\noindent Again, fixing $\omega<728,\alpha=6\cdot10^{12},T=\gamma
_{22000000}<10379599.7274$ and $14<\gamma_{1}\leq\gamma<10379600$ we get from
(\ref{RieleS1}) and (\ref{RieleS2})%
\begin{align*}
\left\vert \frac{\partial}{\partial x}s\left(  \alpha,\omega,x\right)
\right\vert  &  \leq\frac{1508.2}{x},\\
\left\vert \frac{\partial}{\partial x}t\left(  \alpha,\omega,x\right)
\right\vert  &  \leq\frac{2.14625}{x^{2}}.
\end{align*}
Proceeding as before we estimate
\begin{align*}
\Delta S_{1}  &  \leq2.6011\times10^{-5},\\
\Delta S_{2}  &  \leq4.9599\times10^{-11}.
\end{align*}
With $\omega=727.95134$, $\eta=0.000022833$ and $A=6.85\cdot10^{7}$ our
computational value for the $S^{\ast}$ sum gives $S^{\ast}\left(
\alpha,\omega,T\right)  \geq-1.002922947193156$, which again is near at the
Saouter-Demichel machine value with size $-1.002906086981405$. Then

\begin{table}[H]
\begin{center}
{
\begin{tabular}
[c]{|r|c|c|}\hline
\rowcolor{orange} $\eta$ & $R$ & Lower estimate for $I\left(  \omega
,\eta\right)  $\\\hline
$2A/\alpha=0.0000228$ & $2.79507 \times10^{-3}$ & $0.000101863$\\
$0.0000200$ & $2.79503 \times10^{-3}$ & $0.000101907$\\
$0.0000180$ & $2.79500 \times10^{-3}$ & $0.000101939$\\
$0.0000160$ & $2.79527 \times10^{-3}$ & $0.000101670$\\
$0.0000159$ & $2.83095 \times10^{-3}$ & $0.000015987$\\
$0.0000158$ & $6.97516 \times10^{-3}$ & negative\\
$0.0000155$ & $5.15554 \times10^{+1}$ & negative\\\hline
\end{tabular}
}
\end{center}
\caption{Resizing $\eta$ in the Saouter-Demichel region.}%
\label{ResizeSaouter}%
\end{table}

\noindent Comparing to the original estimate for $I\left(  \omega,\eta\right)
$ by Saouter-Demichel with calculated value $0.00003751696746$ we perform much
better. Moreover, the best resized value $\eta=0.0000159$ is improved by about
$30.80$ percent compared to the original $\eta$.

\subsection{Runs of successive integers}

From our better estimates for $I\left(  \omega,\eta\right)  $ in Table
\ref{ResizeBays} and Table \ref{ResizeSaouter} we can further deduce larger
ranges of consecutive integers for which the difference $\pi\left(  x\right)
-li\left(  x\right)  $ remains positive. To this end, denote $\delta$ a fixed
positive real number for which we have
\begin{align*}
I\left(  \omega,\eta\right)   &  =\int_{\omega-\eta}^{\omega+\eta}K\left(
u-\omega\right)  ue^{-u/2}\left(  \pi\left(  e^{u}\right)  -li\left(
e^{u}\right)  \right)  du\\
&  =\int_{\omega-\eta}^{\omega+\eta}K\left(  u-\omega\right)  F\left(
u\right)  du>\delta.
\end{align*}
Then%
\[
\sup_{u\in\left[  \omega-\eta,\omega+\eta\right]  }F\left(  u\right)  \geq
\int_{\omega-\eta}^{\omega+\eta}K\left(  u-\omega\right)  F\left(  u\right)
du>\delta,
\]
from which we deduce that for some $u\in\left[  \omega-\eta,\omega
+\eta\right]  $ we have $F\left(  u\right)  \geq\delta$. Thus, since
$e^{u/2}/u$ is an increasing function for $u>1$, we find that%
\[
\pi\left(  e^{u}\right)  -li\left(  e^{u}\right)  \geq\delta e^{\frac
{\omega-\eta}{2}}\frac{1}{\omega-\eta}.
\]
Now, let $x\in\mathbb{R}$ with $\pi\left(  x\right)  -li\left(  x\right)  \geq
A>0$. Then for all $0<y<A\log x$ we still have $\pi\left(  x+y\right)
-li\left(  x+y\right)  >0$, since%
\begin{align*}
&  \pi\left(  x+y\right)  -li\left(  x+y\right) \\
&  =\pi\left(  x+y\right)  -\pi\left(  x\right)  +\pi\left(  x\right)
-li\left(  x\right)  +li\left(  x\right)  -li\left(  x+y\right) \\
&  \geq0+A-\int_{x}^{x+y}\frac{dt}{\log t}\\
&  \geq A-y\frac{1}{\log x} >0.
\end{align*}
If we take $A=\delta e^{\frac{\omega-\eta}{2}}\frac{1}{\omega-\eta}$ and
$x=e^{u}$ for some $u\in\left[  \omega-\eta,\omega+\eta\right]  $ we find
that, for $0<y<A\left(  \omega-\eta\right)  $,
\[
\pi\left(  x+y\right)  -li\left(  x+y\right)  \geq A-\frac{y}{\log e^{u}}\geq
A-\frac{y}{\omega-\eta}>0.
\]
Thus we have a run of consecutive integers from $x$ to $x+\delta
e^{\frac{\omega-\eta}{2}}$ for which the difference $\pi\left(  x\right)
-li\left(  x\right)  $ is positive in the vicinity of $x=e^{\omega}$. So we
can state the following additional theorem.

\begin{theorem}
In the vicinity of $e^{727.952018}$ there are more than $4.61877
\times10^{154}$ successive integers where the difference $\pi\left(  x\right)
-li\left(  x\right)  $ is positive.
\end{theorem}

\section{The $R_{5}$ error term in the Saouter-Trudgian-Demichel result}

\label{Trudgian}

Since the error estimate for the relevant $R_{5}$ term in
(\cite{SaouterandDemichelandTrudgian}, Theorem 3.1) follows the same bounding
strategy as in our proof we present the corresponding result on the $R_{5}$
term. However, we do not present a detailed proof which can be established by standard arguments following the subsequent steps.

\begin{theorem}
\label{SaouterTrudgianDemichel}
Let $A$ be a positive number such that $\beta=\frac{1}{2}$ for all zeros
$\rho=\beta+i\gamma$ of $\zeta\left(  s\right)  $ for which $0<\gamma\leq A$.
Let $\alpha,\eta$ and $\omega$ be positive numbers such that $\omega>73.69$
and the conditions%
\begin{align*}
\frac{5A}{4\omega}  &  \leq\alpha\leq A^{2}\text{,}\\
0  &  <\eta<\frac{\omega}{100}%
\end{align*}
hold. Let%
\begin{align*}
K\left(  x\right)   &  =\sqrt{\frac{\alpha}{2\pi}}e^{-\frac{\alpha}{2}x^{2}%
},\\
I\left(  \omega,\eta\right)   &  =\int_{\omega-\eta}^{\omega+\eta}K\left(
u-\omega\right)  \frac{ue^{-u/2}}{1+\frac{2}{u}+\frac{10.04}{u^{2}}}\left(
\pi\left(  e^{u}\right)  -li\left(  e^{u}\right)  \right)  du.
\end{align*}
Then, for $2\pi e<T\leq A,$%
\begin{align*}
I\left(  \omega,\eta\right)   &  =-1+\sum_{\substack{\rho\\0<\left\vert
\gamma\right\vert \leq T}}e^{i\gamma\omega}\left(  \frac{1}{\rho}+\frac
{1}{\omega\rho^{2}}\right)  e^{-\gamma^{2}/2\alpha}+\\
&  R_{1}-R_{2}-R_{3}-R_{4}-R_{5}%
\end{align*}
where%
\begin{align*}
R_{1}  &  =\frac{2}{\sqrt{\alpha}}K\left(  \eta\right)  ,\\
R_{2}  &  =\left(  \omega+\eta\right)  \left(  \log2e^{-\left(  \omega
-\eta\right)  /2}+3e^{-\left(  \omega-\eta\right)  /6}\right)  ,\\
R_{3}  &  =0.19K\left(  \eta\right)  +\frac{0.35}{\omega^{2}\sqrt{\alpha}%
}\left(  \log^{2}\frac{A}{2\pi}+11.81\right)  +\frac{0.00292}{\left(
\omega-\eta\right)  ^{2}},\\
R_{4}  &  =e^{-\frac{T^{2}}{2\alpha}}\left(  \frac{\alpha}{\pi T^{2}}\log
\frac{T}{2\pi}+\frac{8\log T}{T}+\frac{4\alpha}{T^{3}}\right)  \left(
1+\frac{1}{\omega T}\right)  ,\\
R_{5}  &  =e^{\frac{\omega+\eta}{2}}\left(  1+\frac{22}{A\omega}\right)  A\log
A\left(  \frac{13.840}{A}e^{-\frac{\alpha}{4}\eta^{2}}+11.951\eta
e^{-\frac{A^{2}}{2\alpha}}\right)  .
\end{align*}
If the Riemann hypothesis holds, the factor $e^{\left(  \omega+\eta\right)
/2}$ in $R_{5}$ can be replaced by $1$.
\end{theorem}

\begin{remark} Since Theorem \ref{SaouterTrudgianDemichel} is applied together with the same double summation formula as in Theorem \ref{theorem31} the error terms $R_{1}$, $R_{2}$, $R_{3}$ and $R_{4}$ remain unchanged in their original representation.
\end{remark}

\begin{proof}
Following the procedure as in Lehman we need an estimate for
\[
\sum_{\substack{\rho\\\left\vert \gamma\right\vert >A}}\int_{\omega-\eta
}^{\omega+\eta}K\left(  u-\omega\right)  \frac{ue^{-\frac{u}{2}}li\left(
e^{u\rho}\right)  }{1+\frac{2}{u}+\frac{10.04}{u^{2}}}du.
\]
We start with
\[
E=\left\{  z\in\mathbb{C}:-\frac{3}{10}\pi\leq\arg\left(  z\right)  \leq
\frac{3}{10}\pi,\quad\left\vert z\right\vert \geq14\right\}
\]
together with the function $f:E\rightarrow\mathbb{C}$ by defining
\[
f_{\rho}\left(  z\right)  =\rho ze^{-\rho z}li\left(  e^{\rho z}\right)
\frac{1}{1+\frac{2}{z}+\frac{10.04}{z^{2}}}e^{-\frac{\alpha}{2}\left(
z-\omega\right)  ^{2}}.
\]
We may write
\begin{align}
&  \int_{\omega-\eta}^{\omega+\eta}K\left(  u-\omega\right)  \frac
{ue^{-\frac{u}{2}}li\left(  e^{u\rho}\right)  }{1+\frac{2}{u}+\frac
{10.04}{u^{2}}}du\nonumber\\
&  =\sqrt{\frac{\alpha}{2\pi}}\frac{1}{\rho}\int_{\omega-\eta}^{\omega+\eta
}e^{u\left(  \rho-\frac{1}{2}\right)  }f_{\rho}\left(  u\right)
du.\label{Trudgian1}%
\end{align}
Performing partial integration for (\ref{Trudgian1}) we obtain for all $N\geq1$
the estimate
\begin{align}
&  \left\vert \int_{\omega-\eta}^{\omega+\eta}e^{u\left(  \rho-\frac{1}%
{2}\right)  }f_{\rho}\left(  u\right)  du\right\vert \nonumber\\
&  \leq e^{\frac{\omega+\eta}{2}}\sum_{n=0}^{N-1}\frac{\left\vert f_{\rho
}^{\left(  n\right)  }\left(  \omega+\eta\right)  \right\vert +\left\vert
f_{\rho}^{\left(  n\right)  }\left(  \omega-\eta\right)  \right\vert
}{\left\vert \gamma\right\vert ^{n+1}}\nonumber\\
&  +e^{\frac{\omega+\eta}{2}}\frac{1}{\left\vert \gamma\right\vert ^{N}}%
\int_{\omega-\eta}^{\omega+\eta}\left\vert f_{\rho}^{\left(  N\right)
}\left(  u\right)  \right\vert du.\label{Trudgian2}%
\end{align}
For all $z\in E$ and further, since $\left\vert \operatorname{Im}\left(
\rho\right)  \right\vert >14.1$ for all non trivial zeta zeros, we have%
\begin{align*}
0.4774\pi &  <\left\vert \arg\left(  \rho\right)  \right\vert <\frac{\pi}%
{2},\\
0 &  \leq\left\vert \arg\left(  z\right)  \right\vert \leq\frac{3}{10}\pi.
\end{align*}
Thus%
\[
\frac{\pi}{6}<0.1774\pi\leq\left\vert \arg\left(  \rho z\right)  \right\vert
\leq\frac{4}{5}\pi<\frac{5}{6}\pi.
\]
from which we obtain that $f_{\rho}$ is an analytic function in the sector
plane $E$ provided that the polynomial $z^{2}+2z+10.04$ is analytic in $E$.
Since $z^{2}+2z+10.04$ has roots at $z_{0}$ and $\overline{z_{0}}$ with
$z_{0}=-1+i\sqrt{9.04}$ which are outside of $E$, this is certified. Using
$\left\vert \operatorname{Im}\left(  \rho z\right)  \right\vert =\left\vert
\rho z\right\vert \left\vert \sin\left(  \arg\left(  \rho z\right)  \right)
\right\vert \geq\left\vert \rho z\right\vert \left\vert \sin\frac{\pi}%
{6}\right\vert =\frac{1}{2}\left\vert \rho z\right\vert $, $\left\vert \rho
z-t\right\vert ^{2}\geq\left\vert \operatorname{Im}\left(  \rho z\right)
\right\vert $ together with $\left\vert z-z_{o}\right\vert \geq\left\vert
z\right\vert -\left\vert z_{0}\right\vert \geq14-\sqrt{10.04}>0$ we obtain by
(\ref{LiLehman}) with $n=1$ the estimate
\begin{align}
\left\vert f_{\rho}\left(  z\right)  \right\vert  &  \leq\left(  1+\frac
{4}{\left\vert \gamma\right\vert \left\vert z\right\vert }\right)  \left\vert
\frac{z^{2}}{z^{2}+2z+10.04}\right\vert \left\vert e^{-\frac{\alpha}{2}\left(
z-\omega\right)  ^{2}}\right\vert \nonumber\\
&  \leq\left(  1+\frac{4}{\left\vert \gamma\right\vert \left\vert z\right\vert
}\right)  \left(  1+\frac{\sqrt{10.04}}{\left\vert z\right\vert -\sqrt{10.04}%
}\right)  ^{2}\left\vert e^{-\frac{\alpha}{2}\left(  z-\omega\right)  ^{2}%
}\right\vert \nonumber\\
&  \leq1.671\left(  1+\frac{4}{\left\vert \gamma\right\vert \left\vert
z\right\vert }\right)  \left\vert e^{-\frac{\alpha}{2}\left(  z-\omega\right)
^{2}}\right\vert .\label{Trudgian3}%
\end{align}
Estimating the terms $f_{\rho}^{\left(  n\right)  }$ with our contour integral
together with (\ref{frhoestimate3}), (\ref{frhoestimate4}) and (\ref{Trudgian3}) we
get
\[
\left\vert f_{\rho}^{\left(  n\right)  }\left(  x\right)  \right\vert
\leq1.671\left(  1+\frac{4}{\left\vert \gamma\right\vert \left\vert
x-r\right\vert }\right)  \frac{n!}{r^{n}}e^{\frac{\alpha}{2}r^{2}}%
e^{-\frac{\alpha}{4}\left(  x-\omega\right)  ^{2}}%
\]
With $\left\vert x-r\right\vert \geq x-r\geq\omega-\eta-r>\omega-\frac{\omega
}{100}-\frac{4}{5}\omega=\frac{19}{100}\omega$ $\geq14$ for $\omega\geq73.69$
and $\left\vert \operatorname{Im}\left(  \rho\right)  \right\vert =\left\vert
\gamma\right\vert >A$ we arrive at%
\begin{align}
\left\vert f_{\rho}^{\left(  n\right)  }\left(  \omega\pm\eta\right)
\right\vert  &  \leq\left\{
\begin{array}
[c]{ll}%
1.671\left(  1+\frac{22}{A\omega}\right)  e^{-\frac{\alpha}{4}\eta^{2}}, &
n=0,\\
1.671\left(  1+\frac{22}{A\omega}\right)  n!\left(  \frac{\alpha e}{n}\right)
^{\frac{n}{2}}e^{-\frac{\alpha}{4}\eta^{2}}, & n=1,2,\ldots,N-1,
\end{array}
\right.  \label{Trudgian4}\\
\left\vert f_{\rho}^{\left(  N\right)  }\left(  x\right)  \right\vert  &
\leq1.671\left(  1+\frac{22}{A\omega}\right)  N!\left(  \frac{\alpha e}%
{N}\right)  ^{\frac{N}{2}}.\label{Trudgian5}%
\end{align}
Now, using (\ref{Trudgian1}), (\ref{Trudgian2}), (\ref{Trudgian4}), (\ref{Trudgian5}) together with (\ref{Sumestimate2}) and (\ref{Sumestimate3}) we obtain the
final result by
\begin{align*}
&  \sum_{\substack{\rho\\\left\vert \gamma\right\vert >A}}\int_{\omega-\eta
}^{\omega+\eta}K\left(  u-\omega\right)  ue^{-\frac{u}{2}}li\left(  e^{u\rho
}\right)  du\\
&  \leq\sqrt{\frac{\alpha}{2\pi}}\sum_{\substack{\rho\\\left\vert
\gamma\right\vert >A}}\frac{1}{\left\vert \gamma\right\vert }\left\vert
\int_{\omega-\eta}^{\omega+\eta}e^{u\left(  \rho-\frac{1}{2}\right)  }f_{\rho
}\left(  u\right)  du\right\vert \\
&  \leq\frac{4}{\sqrt{2\pi}}\left(  1+\frac{22}{Aw}\right)  A\log A\left(
\left(  1+\frac{e}{\sqrt{e}-1}\right)  \frac{1}{A}e^{-\frac{\alpha}{4}\eta
^{2}+\frac{\omega+\eta}{2}}+e\sqrt{e}\eta e^{-\frac{A^{2}}{2\alpha}%
+\frac{\omega+\eta}{2}}\right)  \\
&  \leq\left(  1+\frac{22}{Aw}\right)  A\log A\left(  \frac{13.840}%
{A}e^{-\frac{\alpha}{4}\eta^{2}+\frac{\omega+\eta}{2}}+11.951\eta
e^{-\frac{A^{2}}{2\alpha}+\frac{\omega+\eta}{2}}\right)  \\
&  =R_{5}.
\end{align*}
\end{proof}

\section{Crossovers earlier than $10^{316}$}

\label{Near crossovers}

We give a short discussion on some attractive candidates for crossovers
earlier than $10^{316}$. To this end, for a fixed $T$, we study the function%
\[
f_{T}\left(  \omega\right)  =-1-\sum_{\substack{\rho\\0<\left\vert
\gamma\right\vert \leq T}}\frac{e^{i\omega\gamma}}{\rho},
\]
which is part of the summation in Theorem \ref{theorem31}. We mention that
this sum has also been used for the first time by Lehman (\cite{Lehman}, page
406) for detecting possible crossover regions. Bays and Hudson proceeded in a
similar way for their computed plots in logarithmic scale. While we plot in
logarithmic scale by powers of $10$ we first rescale our test function by
\[
F_{T}\left(  \omega\right)  = f_{T}\left(  \omega\log10 \right)  .
\]
Similar as Bays and Hudson, we compute plots in high resolution for $F_{T}$
with $T=\gamma_{100000}$ and range $10^{20}$ to $10^{320}$. We proceed in two
stages. In a first stage, in Figure \ref{Region20320}, we split the plots in
interval-lengths $\left[  \omega,\omega+20\right]  $ and each interval plot is
computed with 500 plot points. High values (in red) near the zero line appear
to be good candidates for possible crossovers. For comparison we also include
plots for the Bays-Hudson region $\omega\in\left[  300,320\right]  $ and the
te Riele region $\omega\in\left[  360,380\right]  $. In a second stage, in
Figure \ref{EnlargedRegions}, we further present enlarged plots for some hot regions.

\noindent
\includegraphics[width=\textwidth, height=4.0cm]{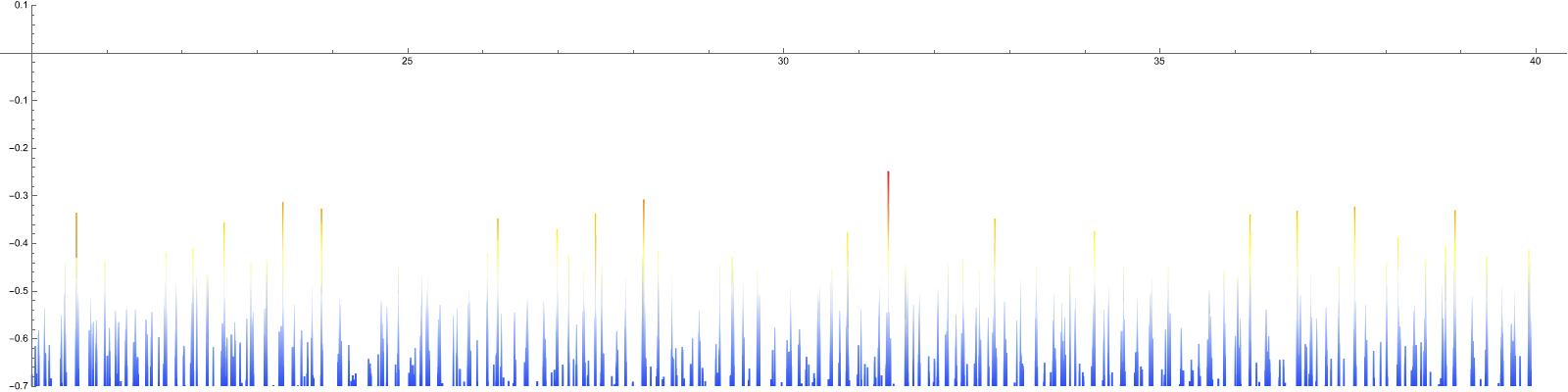}
\includegraphics[width=\textwidth, height=4.0cm]{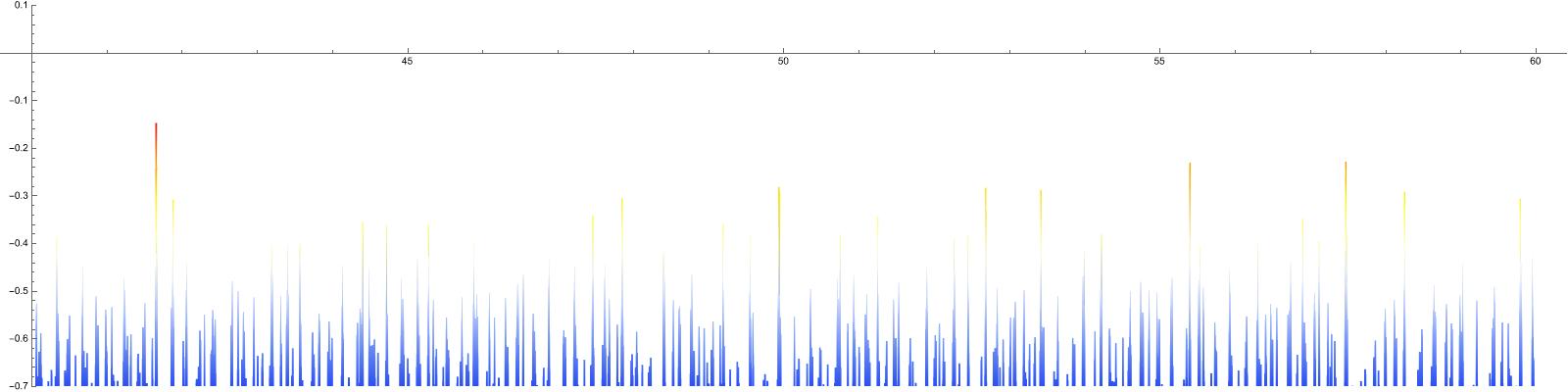}
\includegraphics[width=\textwidth, height=4.0cm]{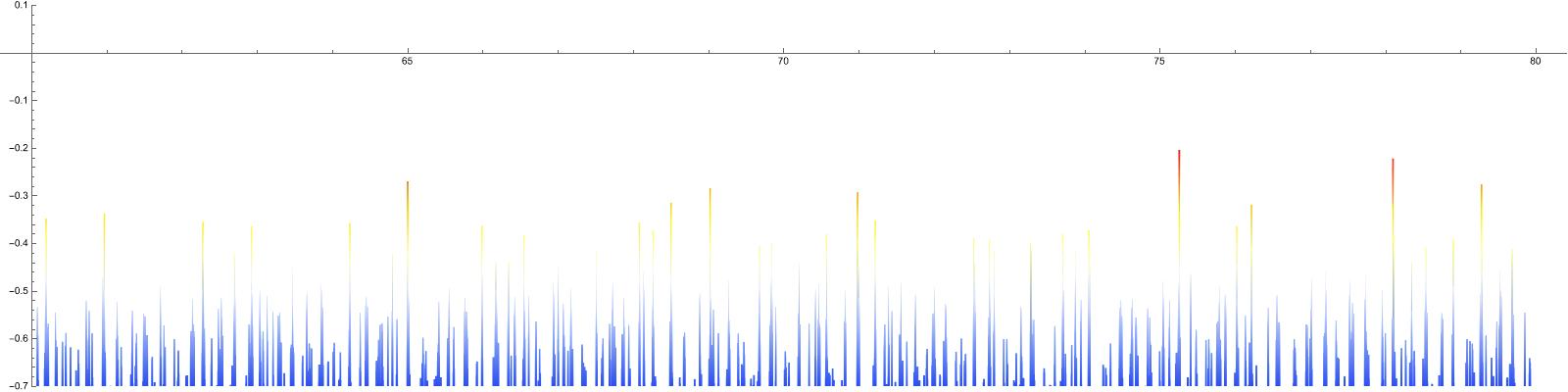}
\includegraphics[width=\textwidth, height=4.0cm]{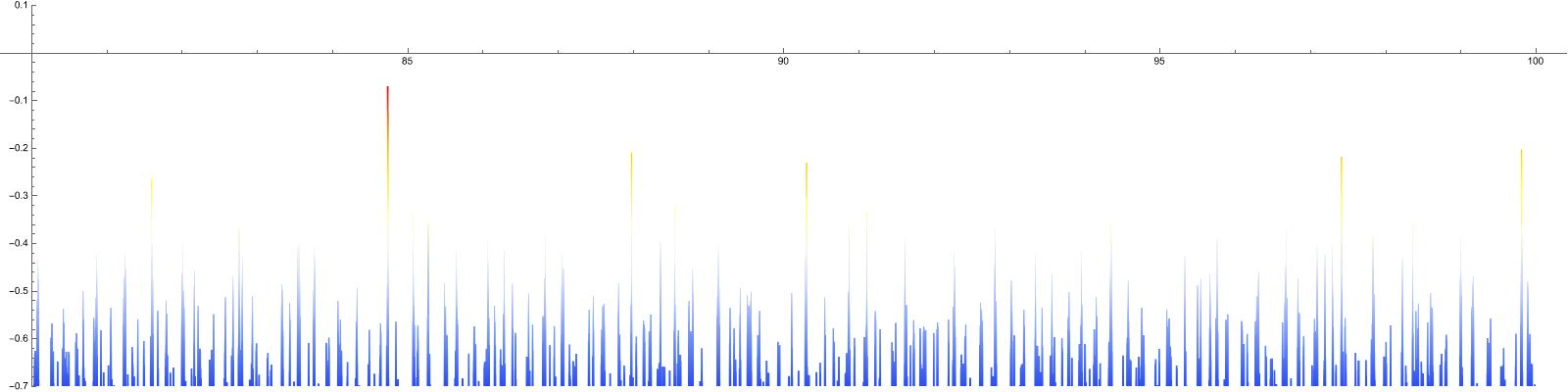}
\includegraphics[width=\textwidth, height=4.0cm]{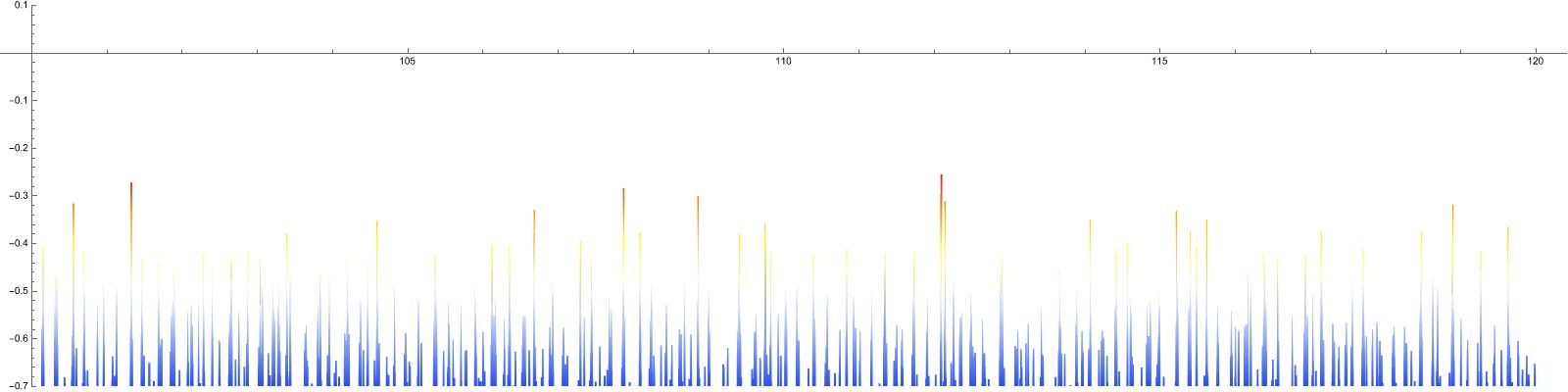}
\includegraphics[width=\textwidth, height=4.0cm]{Mike100_120}
\includegraphics[width=\textwidth, height=4.0cm]{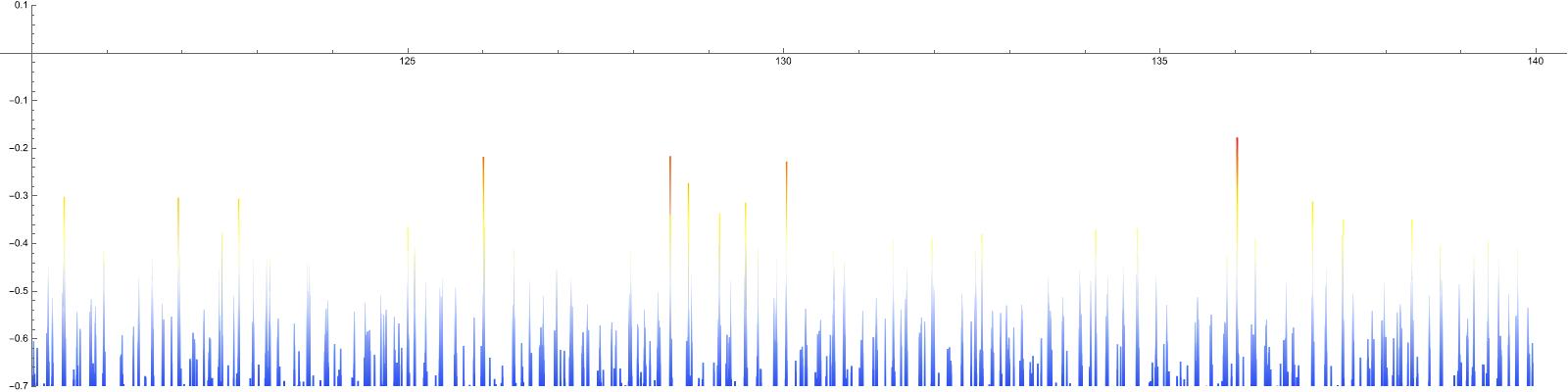}
\includegraphics[width=\textwidth, height=4.0cm]{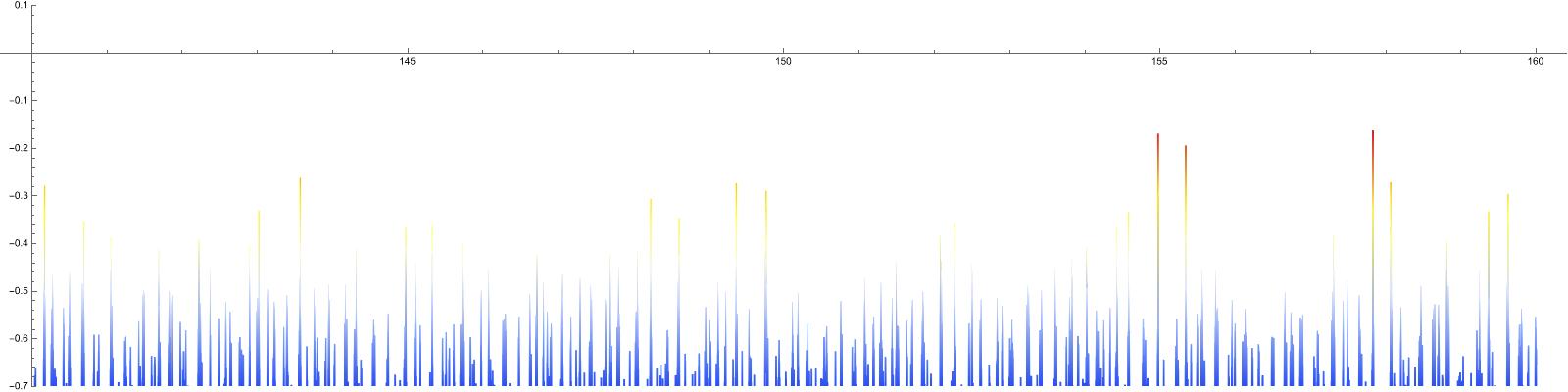}
\includegraphics[width=\textwidth, height=4.0cm]{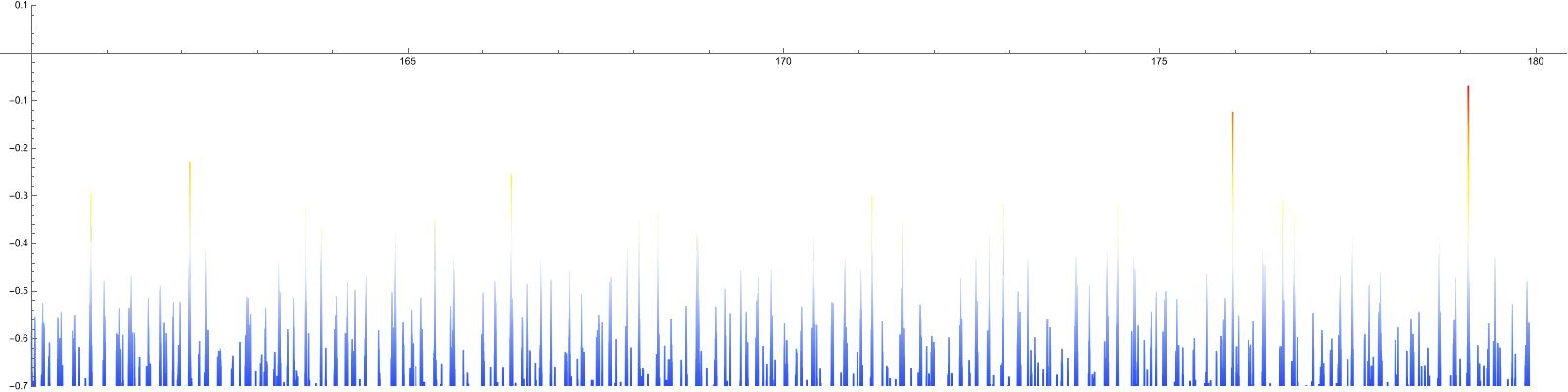}
\includegraphics[width=\textwidth, height=4.0cm]{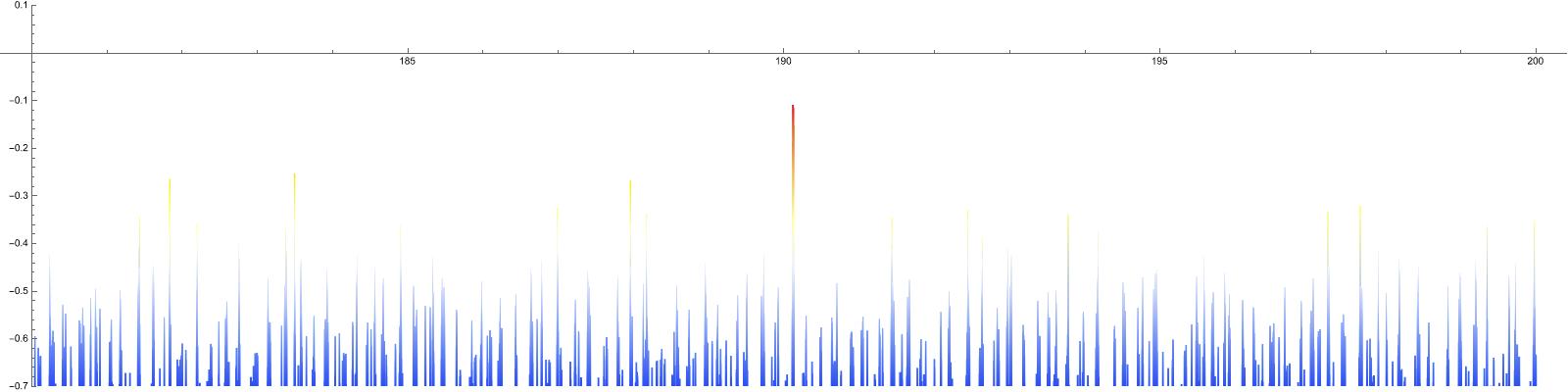}
\includegraphics[width=\textwidth, height=4.0cm]{Mike100_120}
\includegraphics[width=\textwidth, height=4.0cm]{Mike120_140}
\includegraphics[width=\textwidth, height=4.0cm]{Mike140_160}
\includegraphics[width=\textwidth, height=4.0cm]{Mike160_180}
\includegraphics[width=\textwidth, height=4.0cm]{Mike180_200}
\includegraphics[width=\textwidth, height=4.0cm]{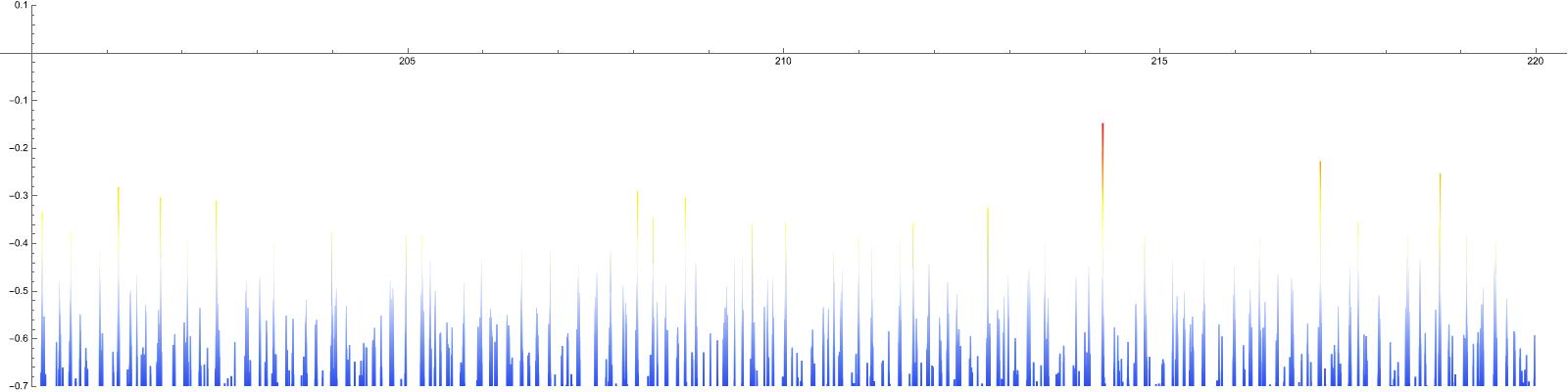}
\includegraphics[width=\textwidth, height=4.0cm]{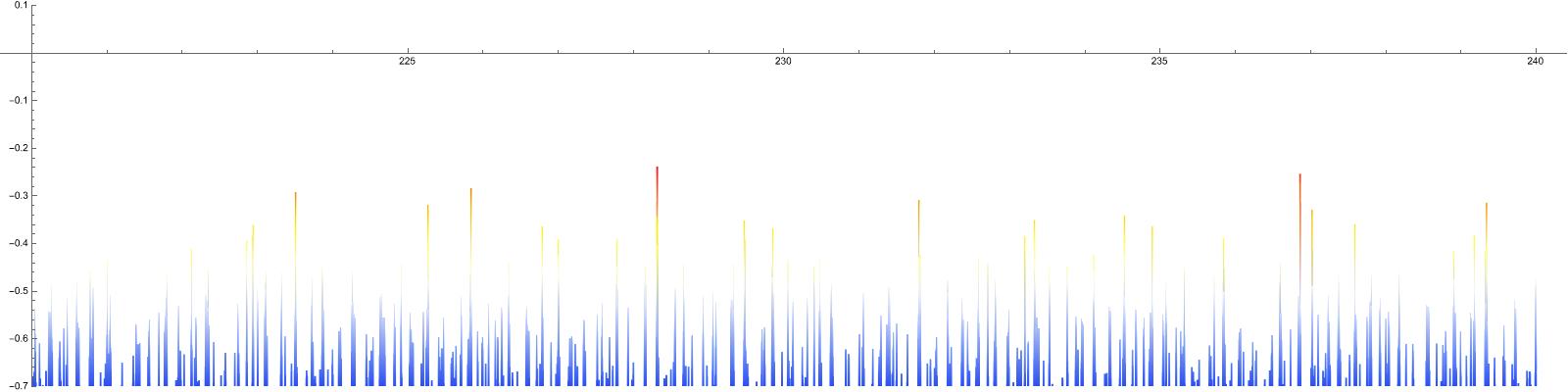}
\includegraphics[width=\textwidth, height=4.0cm]{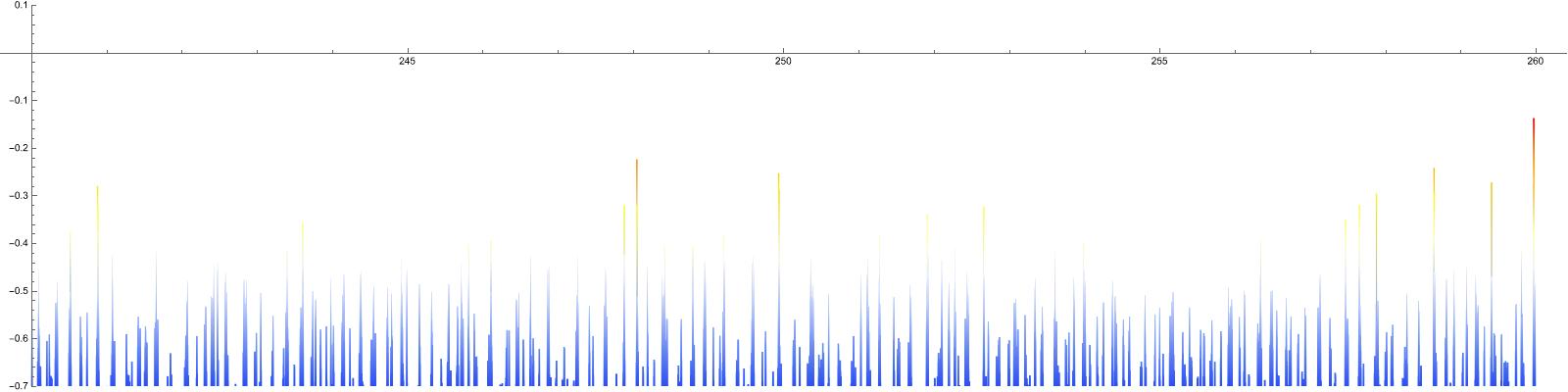}
\includegraphics[width=\textwidth, height=4.0cm]{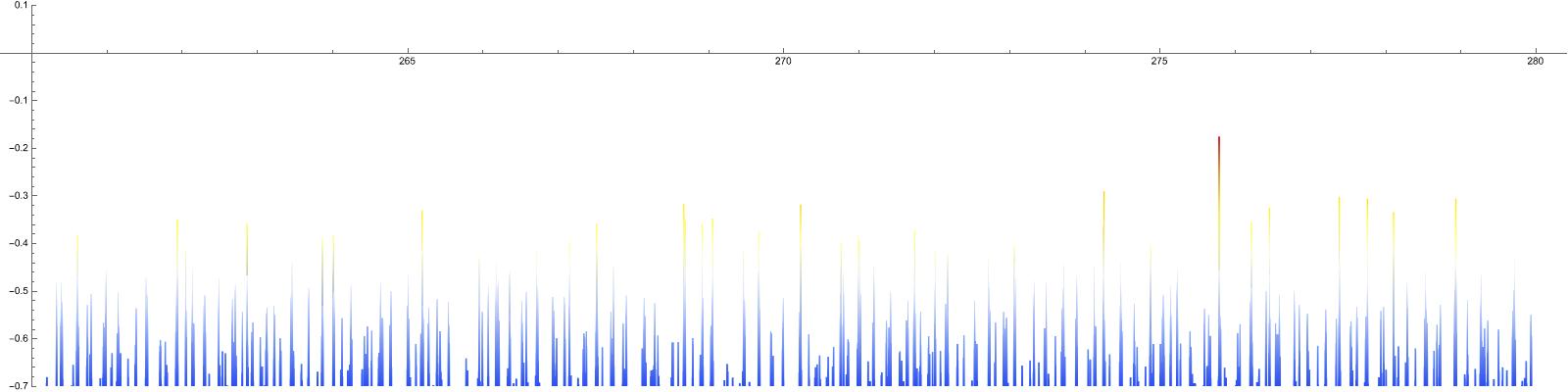}
\includegraphics[width=\textwidth, height=4.0cm]{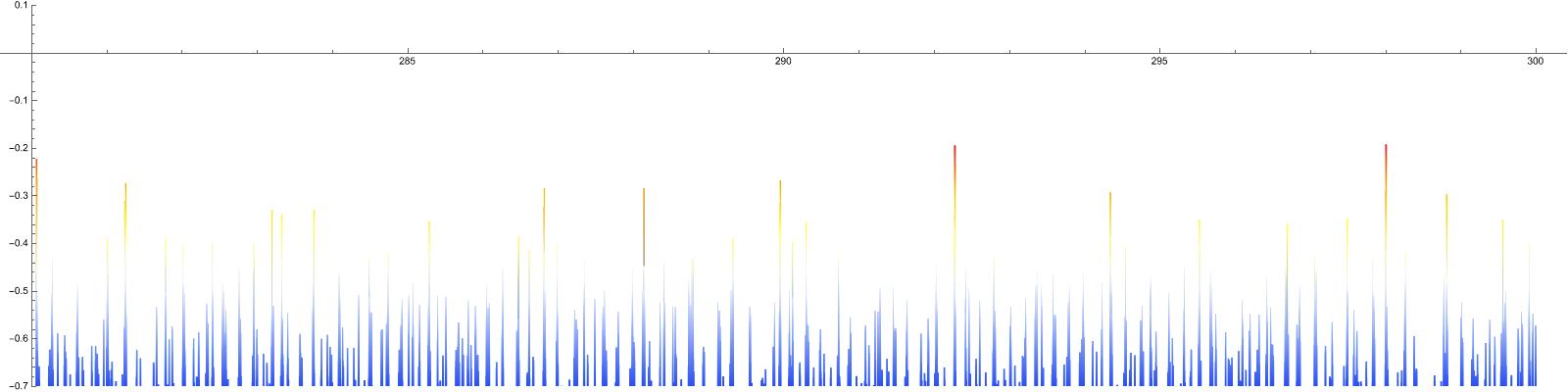}
\includegraphics[width=\textwidth, height=4.0cm]{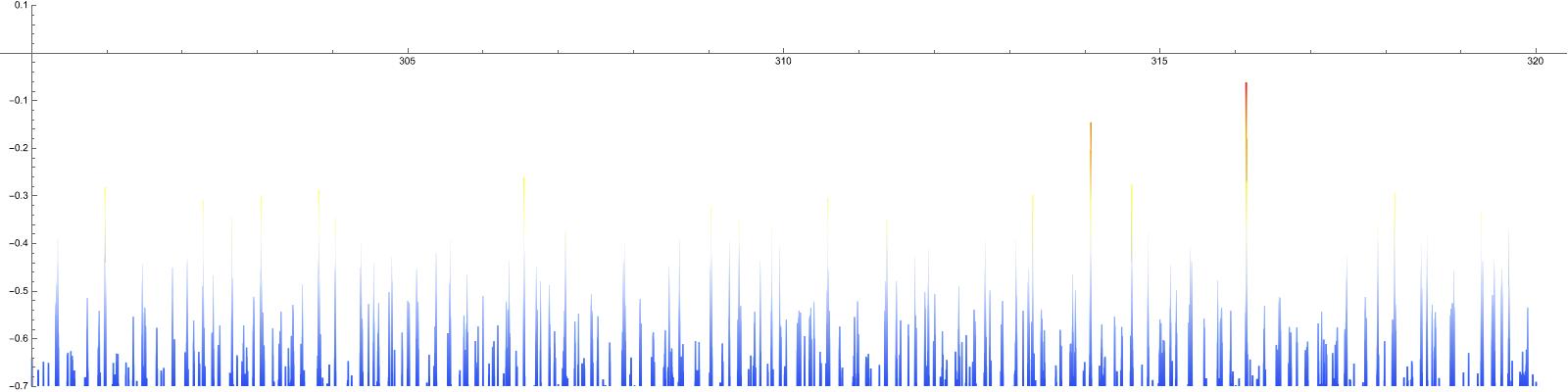}
\includegraphics[width=\textwidth, height=4.0cm]{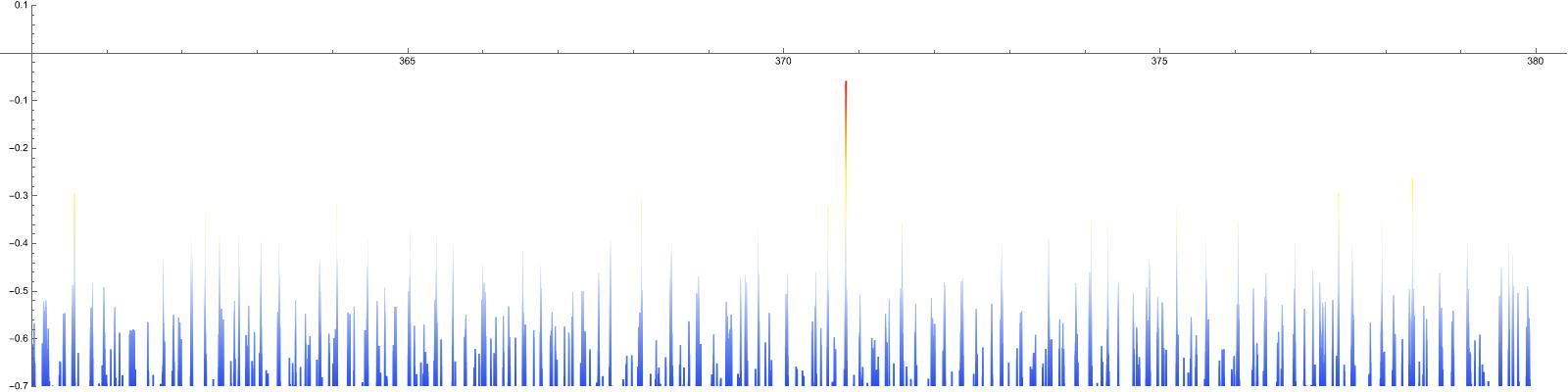}

\begin{figure}[H]
\caption{Regions $10^{20}$ to $10^{320}$, including the te Riele region.}
\label{Region20320}
\end{figure}

\noindent We will give some remarks.

\begin{remark}
The computer plots in Figure \ref{Region20320} should be
viewed very critically. Even with finest resolution of over $500$ plot points
in each subinterval consecutive points are far distant apart in our
logarithmic scale. Therefore, the computer graphics do not necessarily resolve
the real situation, especially concerning the actual height of the local maxima.
\end{remark}

\begin{remark}
The enlarged plots in Figure \ref{EnlargedRegions} are
again computed in high resolution with 500 plot points for $F_{T}$, now with
$T=\gamma_{1000000}$. From these plots we finally collect in Table
\ref{PossibleCrossovers} the following possible candidates for crossover
regions earlier then $10^{316}$. As a reference, we also include plots and sum
values $F_{T}\left(  \omega\right)  $ for the Bays-Hudson region, the te Riele
region and Lehman's region.
\end{remark}

\noindent
\includegraphics[width=\textwidth, height=4.0cm]{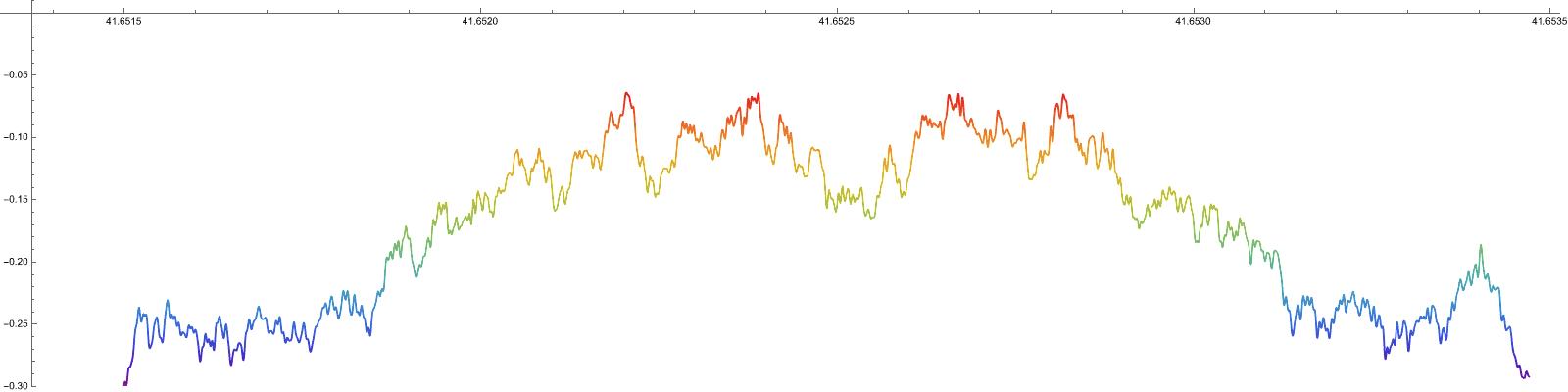}
\includegraphics[width=\textwidth, height=4.0cm]{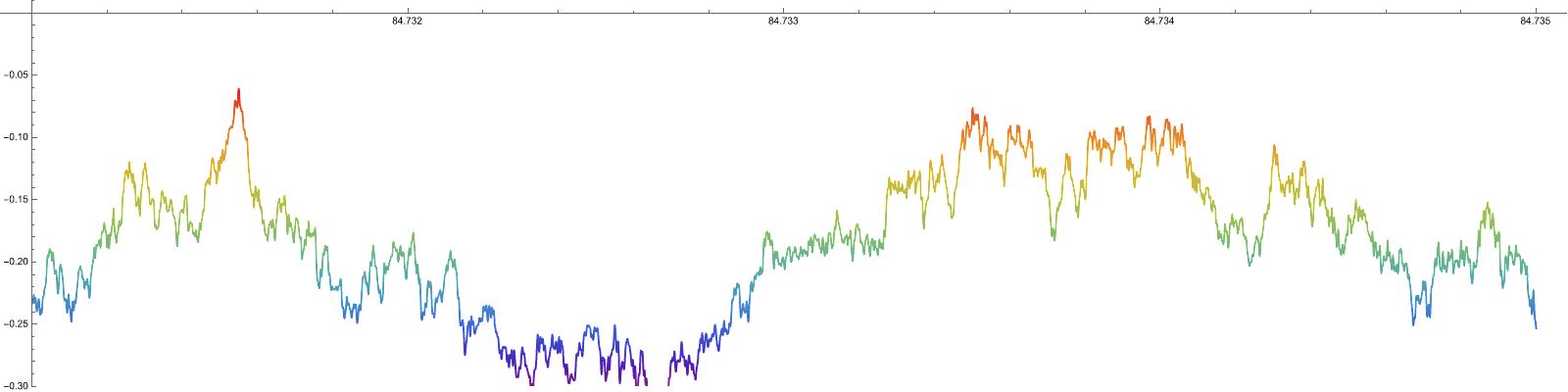}
\includegraphics[width=\textwidth, height=4.0cm]{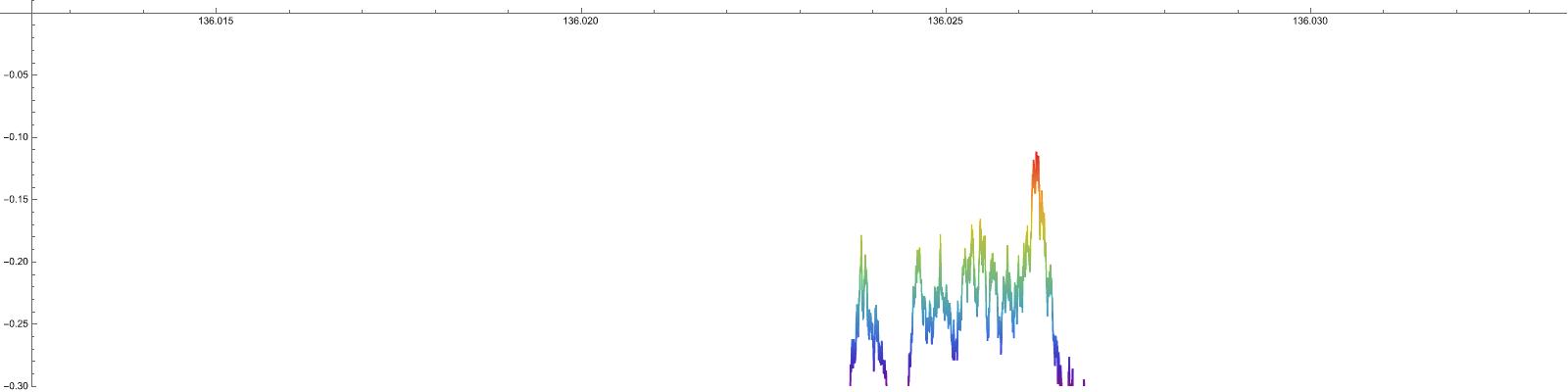}
\includegraphics[width=\textwidth, height=4.0cm]{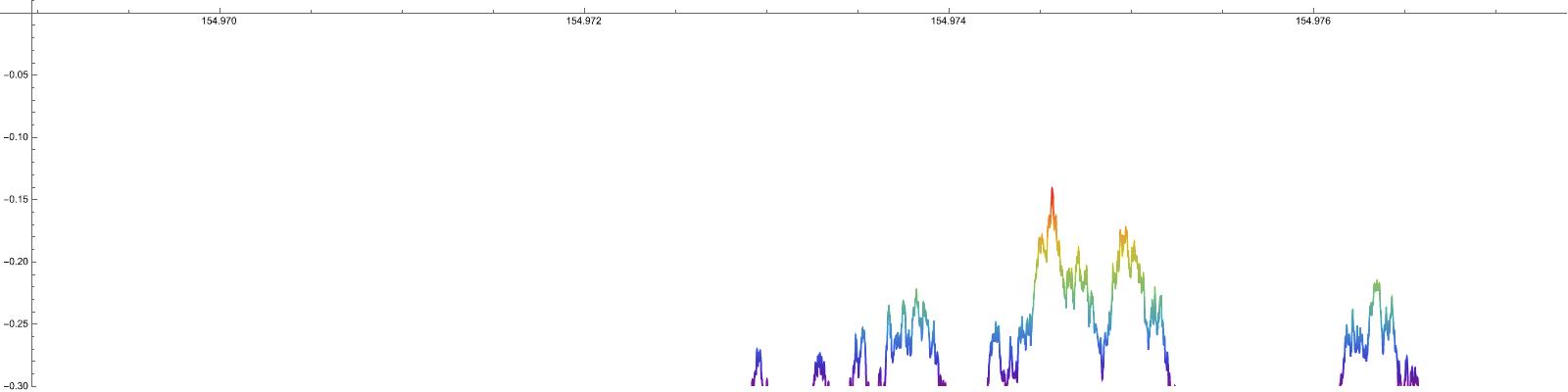}
\includegraphics[width=\textwidth, height=4.0cm]{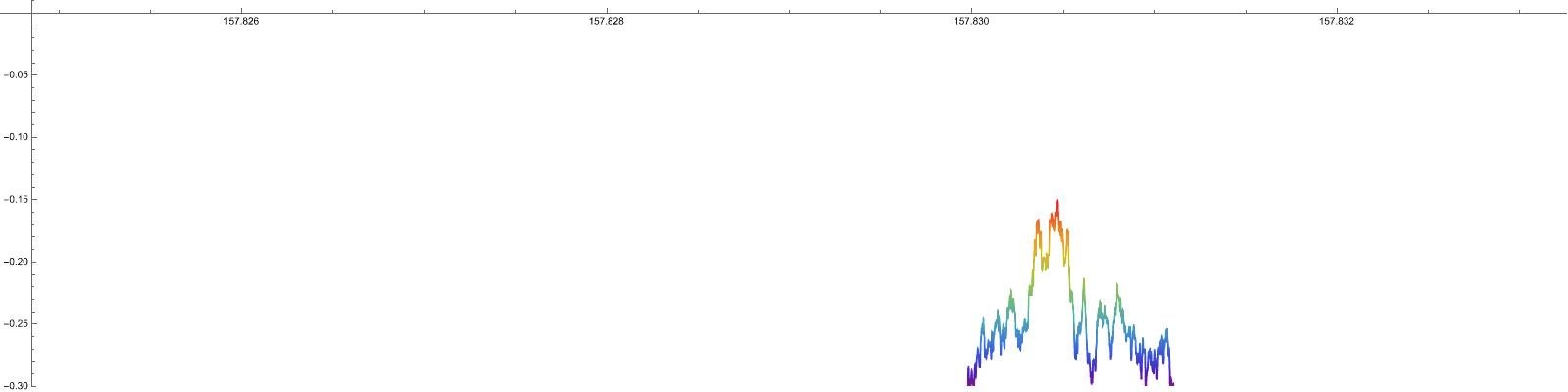}
\includegraphics[width=\textwidth, height=4.0cm]{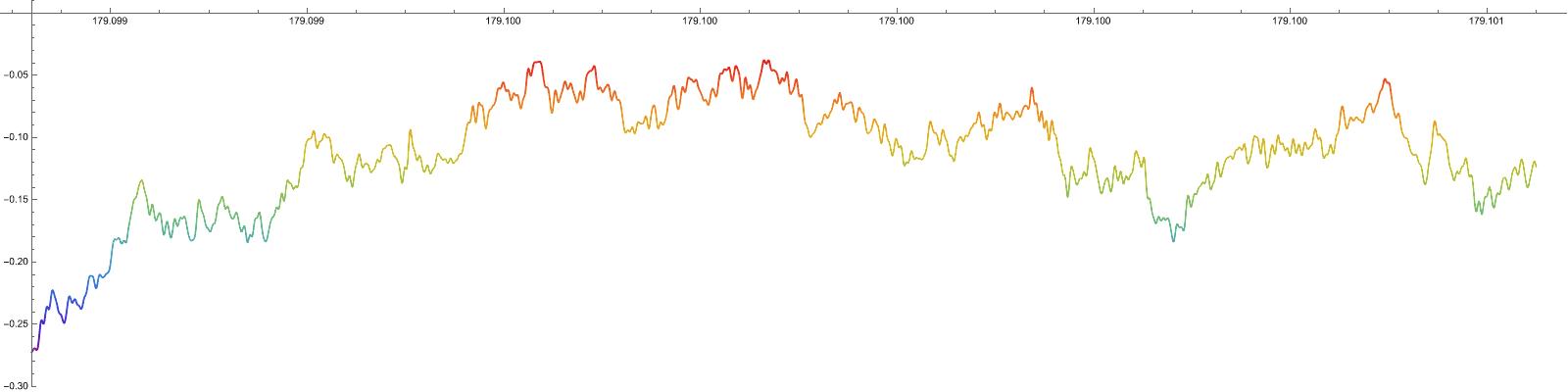}
\includegraphics[width=\textwidth, height=4.0cm]{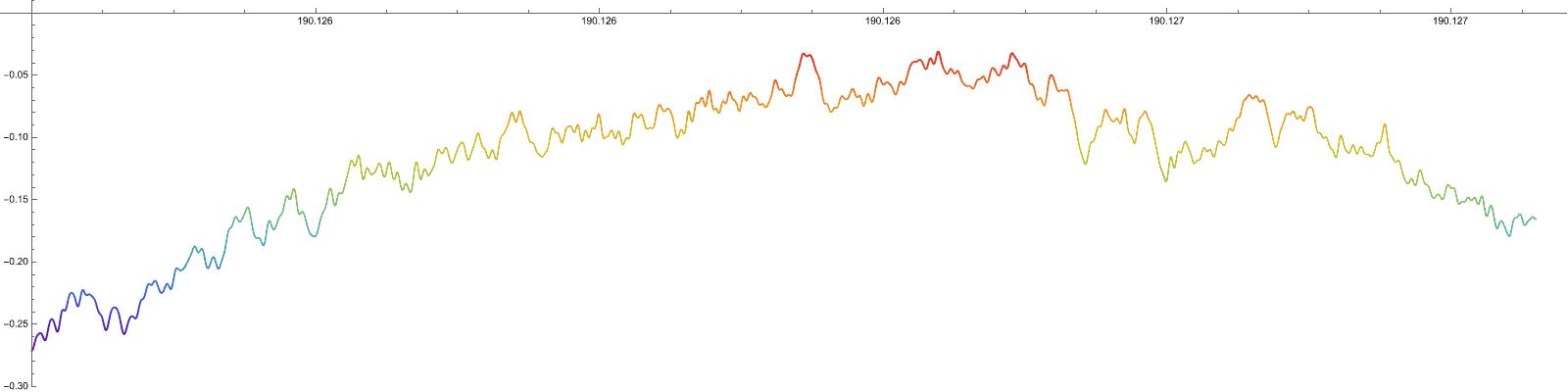}
\includegraphics[width=\textwidth, height=4.0cm]{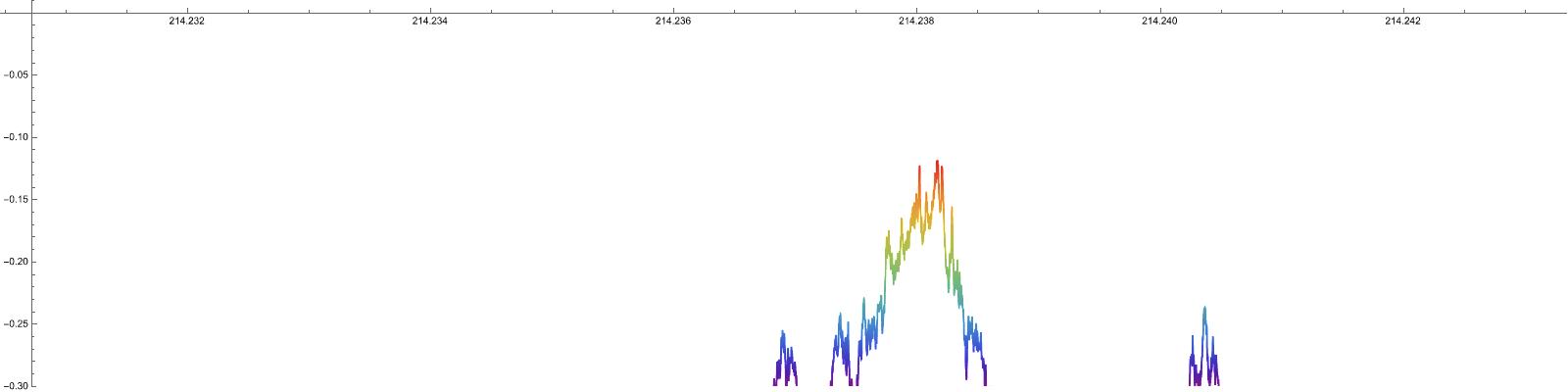}
\includegraphics[width=\textwidth, height=4.0cm]{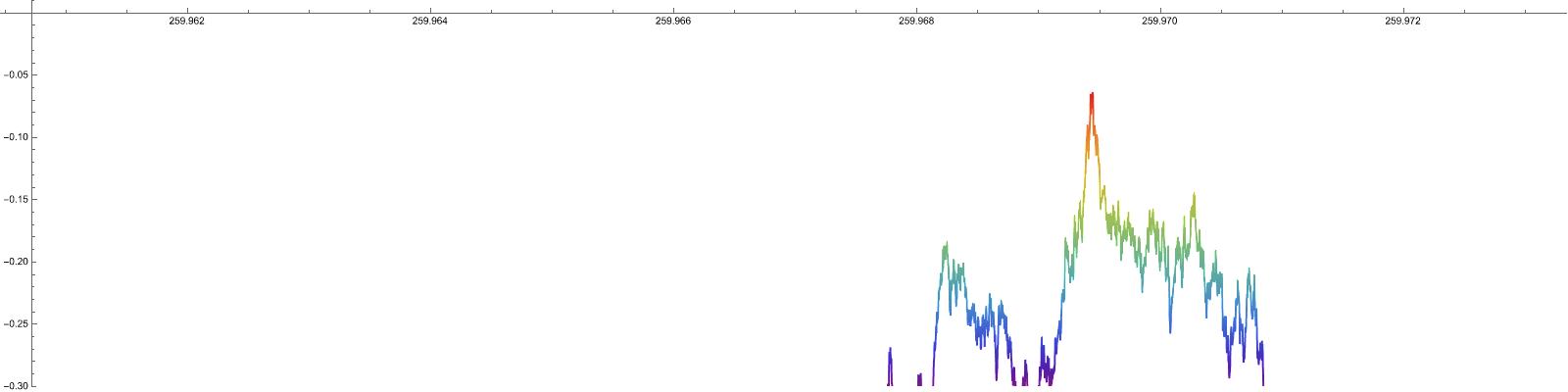}
\includegraphics[width=\textwidth, height=4.0cm]{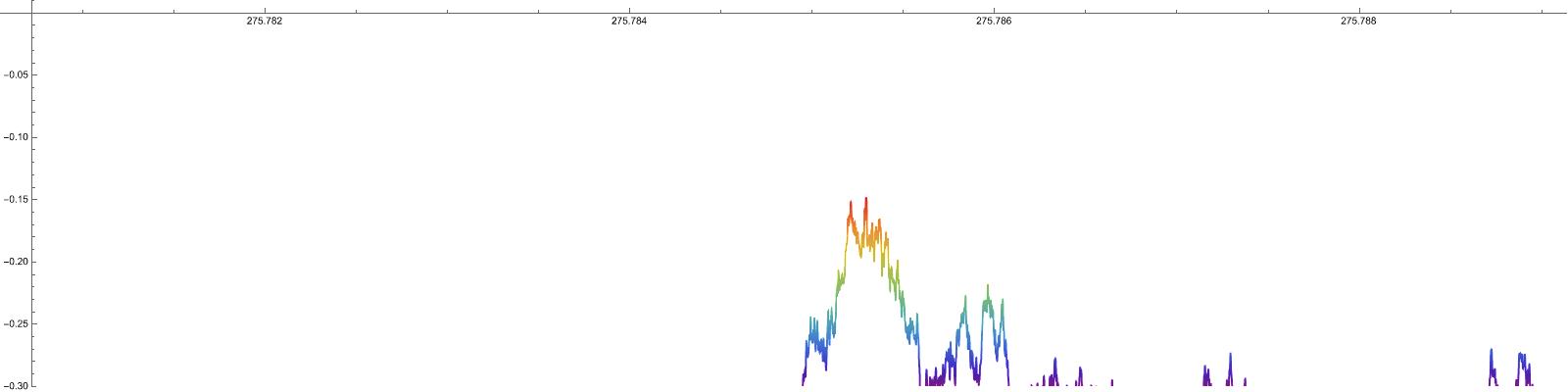}
\includegraphics[width=\textwidth, height=4.0cm]{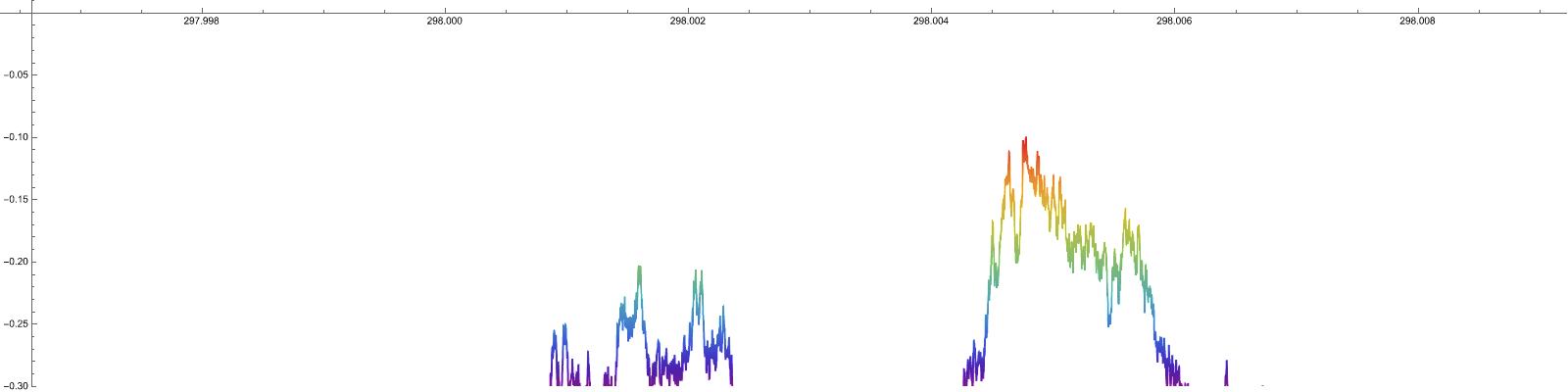}
\includegraphics[width=\textwidth, height=4.0cm]{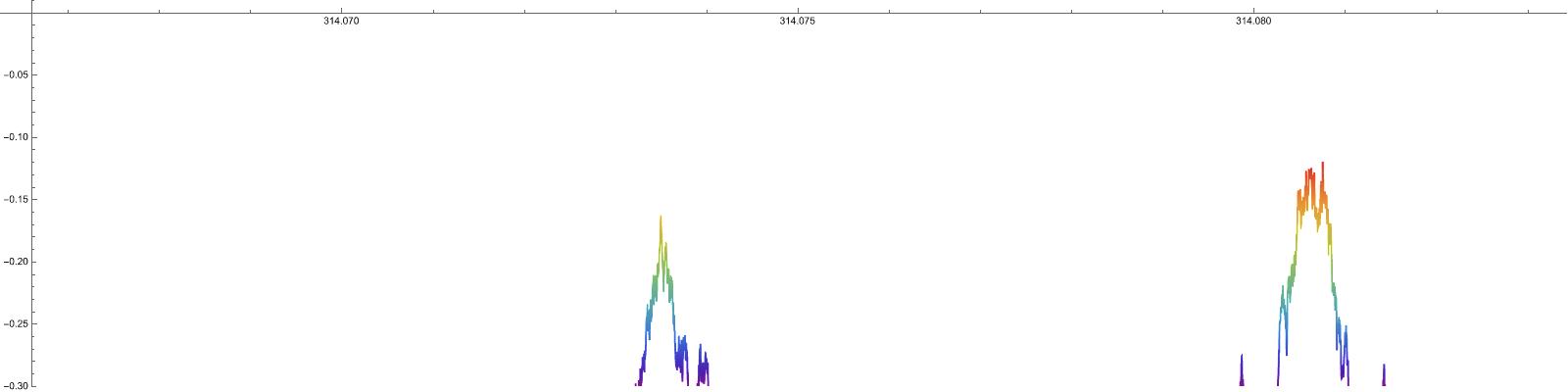}
\includegraphics[width=\textwidth, height=4.0cm]{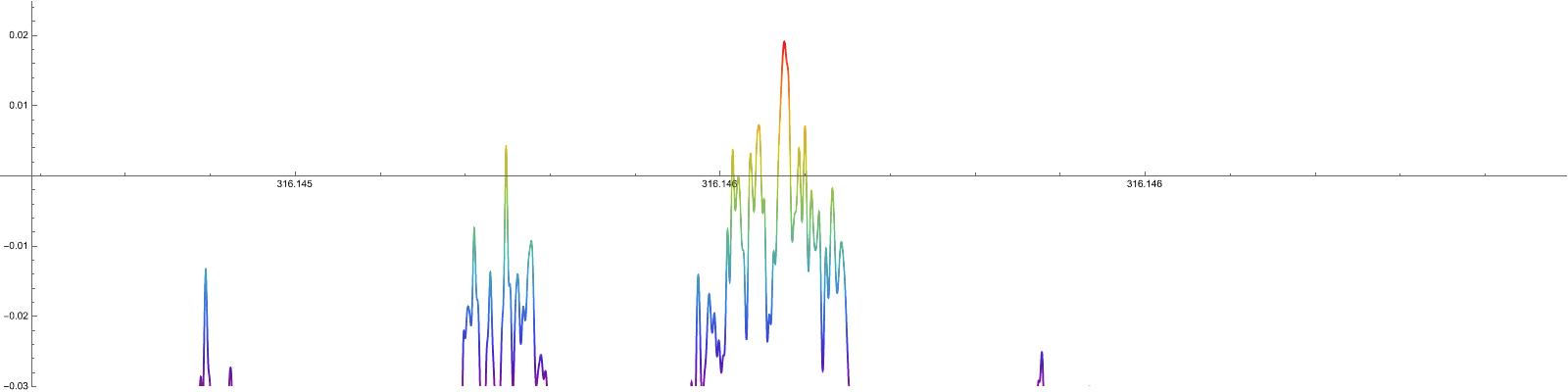}
\includegraphics[width=\textwidth, height=4.0cm]{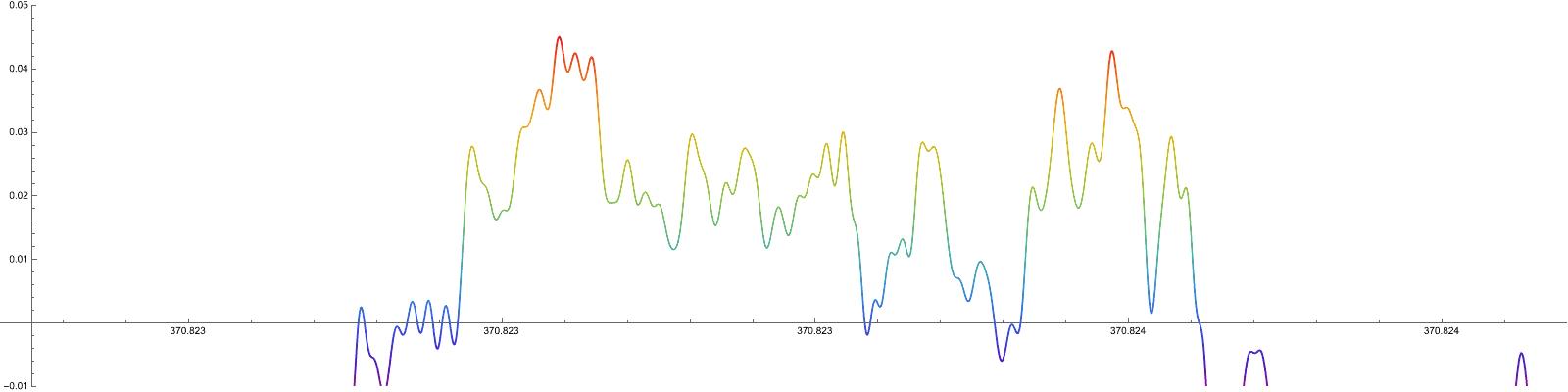}
\includegraphics[width=\textwidth, height=4.0cm]{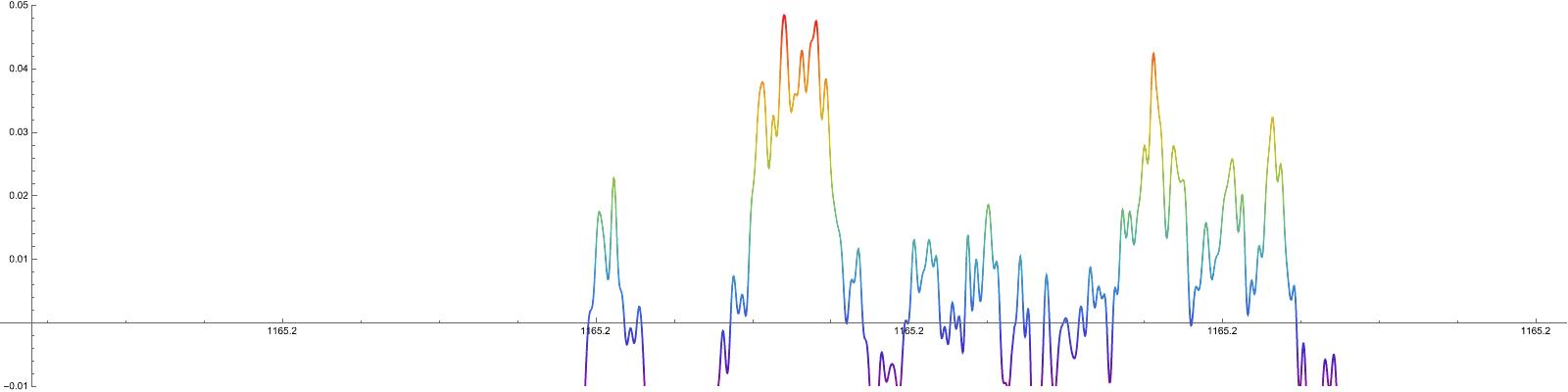}

\begin{figure}[H]
\caption{Enlarged regions for possible crossovers, including the regions of Bays-Hudson, te Riele and Lehman.}
\label{EnlargedRegions}
\end{figure}

\begin{table}[H]
\begin{center}
\begin{tabular}
[c]{|l|c|r|}\hline
\rowcolor{orange} Exponent $\omega$ & $F_{T}\left(  \omega\right)  $ &
Comment\\\hline
$41.6522$ & $-0.0659$ & \\
$84.7316$ & $-0.0597$ & \\
$136.0262$ & $-0.1127$ & \\
$154.9746$ & $-0.1389$ & \\
$157.8305$ & $-0.1502$ & \\
$175.9619$ & $-0.0859$ & Detected by Bays-Hudson\\
$179.0999$ & $-0.0366$ & Detected by Bays-Hudson\\
$190.1264$ & $-0.0313$ & Detected by Bays-Hudson\\
$214.2382$ & $-0.1174$ & \\
$259.9694$ & $-0.0626$ & Detected by Bays-Hudson\\
$275.7852$ & $-0.1478$ & \\
$298.0048$ & $-0.0993$ & Detected by Bays-Hudson\\
$314.0808$ & $-0.1176$ & \\\hline
$316.1456$ & $+0.0195$ & Bays-Hudson region, 2000\\
$370.8233$ & $+0.0453$ & te Riele region, 1987\\
$1165.2019$ & $+0.0489$ & Lehman region, 1966\\\hline
\end{tabular}
\end{center}
\caption{Possible crossovers earlier than $10^{316}$.}%
\label{PossibleCrossovers}%
\end{table}

\noindent If a crossing earlier than $10^{316}$ exists, it is likely to be at
one of the positions in Table \ref{PossibleCrossovers}. It would be very
difficult to prove that there are actual no crossings on these positions.

\medskip\noindent Michael Revers
\newline Department of Mathematics
\newline University Salzburg
\newline Hellbrunnerstrasse 34
\newline5020 Salzburg, AUSTRIA
\newline e-mail: michael.revers@plus.ac.at
\newline $16$-digit ORCID identifier: 0000-0002-9668-5133
\end{document}